\documentclass[12pt]{article} 

\usepackage[utf8]{inputenc} 
\usepackage{amsmath,amsthm,verbatim,amssymb,amsfonts,amscd,graphicx}
\usepackage[colorlinks,citecolor=blue,urlcolor=blue]{hyperref}

\usepackage{geometry} 
\usepackage{graphicx} 

\usepackage{dsfont}

\usepackage{booktabs} 
\usepackage{array} 
\usepackage{verbatim} 
\usepackage{subfig} 
\usepackage{amsmath}
\usepackage[utf8]{inputenc}
\usepackage[english]{babel}
\usepackage{csquotes} 

\usepackage{algorithm2e} 

\usepackage{bm}

\usepackage[symbol]{footmisc}

\usepackage{fancyhdr} 
\usepackage{sectsty}
\allsectionsfont{\sffamily\mdseries\upshape} 

\usepackage[nottoc,notlof,notlot]{tocbibind} 
\usepackage{natbib}
\usepackage{bbm}
\usepackage[titles,subfigure]{tocloft} 
\usepackage{comment}
\usepackage{stmaryrd} 
\usepackage{enumitem} 

\newtheorem{theorem}{Theorem}[section]

\newtheorem{corollary}{Corollary}[section]

\newcommand{\qmq}[1]{\quad\mbox{#1}\quad}

\newcommand{\abs}[1]{\left|{#1}\right|}
\newcommand*{\norm}[1]{\left \lVert {#1} \right \rVert} 

\newcommand*{\Var}[1]{\mathrm{Var} \left [ {#1} \right ]} 
\newcommand*{\Varsub}[2]{\mathrm{Var}_{#1} \left [ {#2} \right ]}
\newcommand*{\Cov}[2]{\mathrm{Cov} \left [ {#1}, {#2} \right ]}
\newcommand*{\E}[1]{\mathbb{E} \left [ {#1} \right ]}

\newcommand*{\Esub}[2]{\mathbb{E}_{#1} \left [ {#2} \right ]}
\newcommand*{\R}{\mathbb{R}} 
\newcommand*{\diag}[1]{\mathrm{diag}({#1})}

\newcommand{\To}{\rightarrow}

\title{Finite-sample bounds to the normal limit under group sequential sampling}
\date{Last modified 13.Feb.23}
\author{\textsc{Julian Aronowitz}$^\dag$ and \textsc{Jay Bartroff}$^*$\\
\small{$^\dag$Google Inc., Culver City, California, USA}\\
\small{$^*$Department of Statistics and Data Sciences, University of Texas at Austin, Austin, Texas, USA}\\
}

\begin{document}
\maketitle

\abstract{In group sequential analysis, data is collected and analyzed in batches until pre-defined stopping criteria are met. Inference in the parametric setup typically relies on the  limiting asymptotic multivariate normality of the repeatedly computed maximum likelihood estimators (MLEs), a result first rigorously proved by \citet{jt} under general regularity conditions. In this work, using Stein’s method we provide optimal order, non-asymptotic bounds on the distance for smooth test functions between the joint group sequential MLEs and the appropriate normal distribution under the same conditions. Our results assume independent observations but allow heterogeneous (i.e., non-identically distributed) data. We examine how the resulting bounds simplify when the data comes from an exponential family.  Finally, we present a general result relating multivariate Kolmogorov distance to smooth function distance which, in addition to extending our results to the former metric, may be  of independent interest.
}

\section{Introduction}\label{sec:intro}

Sequential analysis is a powerful statistical framework in which sample sizes are not determined before conducting a study, and is the dominant statistical methodology in many types of clinical trials \citep{Bartroff13}, among other applications. In clinical trials in particular, data is often collected and analyzed in \emph{groups} until the conditions of a pre-defined stopping criteria are met. In their authoritative textbook on this topic, \citet{sequential_book} note that group sequential analysis can significantly reduce the time required to conduct clinical trials over the traditional fixed sample size designs.

Group sequential analysis is typically  based on repeatedly computed (after each group) maximum likelihood estimators (MLEs), and therefore all the properties of the statistical test (e.g., expected sample size, type~I error probability, power) are functions of the \emph{joint} distribution of these repeatedly computed MLEs. Suppose independent but not necessarily identically distributed observations $Y_1, \dots , Y_n \in \R^t$ are divided into $K$ groups with $n_k$ denoting the number of observations seen up to and including group $k$ for $k = 1, \dots, K$. Let $f_i(y_i ; \theta)$ be the probability mass or density function of $Y_i$ with $\theta \in \Theta \subset \R^d$. We let $\hat{\theta}_k =  \hat{\theta}_k(Y_1, Y_2, \dots, Y_{n_k}) \in \R^{d\times 1}$ be the MLE based on the observations in the first $k$ groups and $\hat{\theta}^K = [\hat{\theta}_1^\intercal, \hat{\theta}_2^\intercal, \dots, \hat{\theta}_K^\intercal]^\intercal \in \R^q$, where $q=dK$, be the group sequential MLE. Here and throughout, $(\cdot)^\intercal$ denotes transpose. Although the exact identification of the distribution of $\hat{\theta}^K$ is often intractable, \cite{jt} showed that it is asymptotically multivariate normal under suitable regularity conditions (see Theorem \ref{thm:MLE sequential asymtotic normality theorem}). This result is perhaps unsurprising in light of the classical result that MLEs are asymptotically normal under the same regularity conditions; see Theorem \ref{thm:MLE asymtotic normality theorem} below.

Natural questions that then arise from this formulation are how many samples are required, and in which groups should the samples be allocated, for the asymptotics to `kick in' and a normal approximation to become appropriate without introducing too much error? One form of an answer to this question would be an upper bound on the error of the multivariate normal approximation of the group sequential MLE. In this work we  provide an optimal order bound on the distance between the group sequential MLE and the appropriate normal distribution under the same regularity conditions with heterogeneous independent data (our Theorem~\ref{my_result_1}). The proof relies on Stein's method and Taylor series arguments. It most closely follows \cite{mvn} in which a similar bound was derived on the multivariate MLE but without the joint, repeatedly computed group sequential structure. Our result generalizes \cite{mvn} since when there is only one group our bound reduces to that found in \cite{mvn}.

In this work we give an upper bound on
$$\abs{\mathbb{E} [h(X)] - \mathbb{E} [ h(Z)]}$$
where $Z$ is the standard multivariate normal, $X$ is the normalized group sequential MLE, and $h$ is a test function in some function class $\mathcal{H}$. In particular for any three times differentiable function $h: \R^q \rightarrow \R$ we denote $\norm{h} := \sup \abs{h}, \norm{h}_1 := \sup_i \abs{\frac{\partial}{\partial x_i} h}, \norm{h}_2 := \sup_{i,j} \abs{\frac{\partial^2}{\partial x_i \partial x_j} h}$, and $\norm{h}_3 := \sup_{i,j, k} \abs{\frac{\partial^3}{\partial x_i \partial x_j \partial x_k} h}$. Then let
\begin{multline} \label{H}
\mathcal{H} = \{h: \R^q \rightarrow \R: h \text{ is three times differentiable with bounded }\\ \norm{h}, \norm{h}_1, \norm{h}_2, \norm{h}_3\}.
\end{multline}

A non-asymptotic bound for the univariate MLE was first developed by \cite{AR_bound} and then expanded to the multivariate case by \cite{mvn}. In \cite{delta} and \cite{mvn_delta}, the bounds for the univariate and multivariate cases respectively were sharpened and simplified under the additional assumption that the MLE follows a special form. Let $\bm{Y} = (\bm{Y_1}, \dots, \bm{Y_n})$ be a random sample of $n$ i.i.d.\ $t$-dimensional random vectors, $\Theta \subset \R^d$, and $\bm{\hat{\theta}}_n(\bm{Y})$ be the resulting MLE. The special form considered by these authors  is 
\begin{equation}\label{eq:mle-special-form}
p \big ( \bm{\hat{\theta}}_n(\bm{Y}) \big ) = \frac{1}{n} \sum_{i = 1}^{n} g(\bm{Y_i}) = \frac{1}{n} \sum_{i = 1}^{n} (g_1(\bm{y_i}), \dots, g_d(\bm{y_i}))^\intercal
\end{equation}
for some $g: \R^t \rightarrow \R^d$ and $p: \Theta \rightarrow \R^d$. One simple example of an MLE that follows this form is the case of independent normal data with unknown mean and variance. \citet{mle_bound_m_dependence} considers data that is $m$-dependent for identically distributed scalar parameter MLEs.  The group sequential MLE is necessarily multivariate and so this particular result will not be generalizable to our setting.

Each of the bounds in \citet{AR_bound}, \citet{mvn}, \citet{delta}, \citet{mvn_delta}, and \citet{mle_bound_m_dependence} are of optimal order $\mathcal{O}(n^{-1/2})$ and defined in terms of slightly different $\mathcal{H}$'s where in general the multidimensional bounds require additional bounded derivatives. All of these results use Stein's method techniques and require $h \in \mathcal{H}$ to be at least bounded and absolutely continuous. \citet{kol_mle_bound} gives a bound of the optimal order on the rate of convergence to normality for univariate MLEs in terms of the Kolmogorov distance, $\mathcal{H} = \{\mathbbm{1}\{x \leq a\}: a \in \R \}$. Unlike the above results this bound does not make use of Stein's method.  Our Corollary~\ref{cor:kol.m3} can be veiwed as a generalization of this to multivariate MLEs and the group sequential setting, and is the first that we are aware of by any method.

Bounds derived derived from characteristic functions \citep{ulyanov1979, ulyanov1986, ulyanov1987} could be of use in our setting, if not for two potential hurdles. First, in the multivariate setting these methods require independence. As one goal here is to develop a method that may later be applicable to observations with some form of  dependence, characteristic function methods appear to be insufficient for that goal. Second, these methods do not provide explicit constants for the bounds produced.

The rest of this paper is organized as follows.  In Section~\ref{sec:bg} we introduce the remaining needed background on MLEs and Stein's method. In Section~\ref{sec:main-result} we present our main result, which we then specialize to exponential family distributions and the exponential distribution in Section~\ref{sec:exp-fam}. And in Section~\ref{sec:mv.K.dist} we present results relating multivariate Kolmogorov distance to smooth function distance, extending our results to the former metric which gives bounds to normality for multivariate MLEs in terms of Kolmogorov distance.

\section{Background}\label{sec:bg}

\subsection{Stein's Method} \label{sec:intro_to_steins_method}
The results  detailed below make heavy use of Stein's method bounds in combination with multivariate Taylor series arguments. We do not attempt to give a complete introduction to Stein's method here, but rather focus on the relevant multivariate results. Readers interested in a more broad introduction are referred to \cite{larrybook}.

Stein's method for multivariate normal approximations hinges on the following multivariate version of the Stein equation. Let $f:\R^q \rightarrow \R$ be twice differentiable and $D^2$ the second derivative, or Hessian matrix, of $f$. Then for $h: \R^q \rightarrow \R$ a multivariate Stein equation is 
\begin{equation*} 
\mathrm{Tr}D^2 f(\bm{w}) - \bm{w}\cdot \nabla f(\bm{w}) = h(\bm{w}) - Nh.
\end{equation*}
Here $Nh$ denotes the expectation $\mathbb{E}h(Z)$, where $Z$ is  a $q$-dimensional standard normal vector.
The following  bound of \citet{size_bias_and_multivariate_solution_bound} relates derivatives of the Stein equation solution~$f$ to those of the test function~$h$:
\begin{equation}\label{properties of the solution multivariate}
\norm{\frac{\partial^k f(\bm{w})}{\prod_{j=1}^{k} \partial w_{i_j}}} \leq \frac{1}{k} \norm{\frac{\partial^k h(\bm{w})}{\prod_{j=1}^{k} \partial w_{i_j}}}, \quad k \geq 1,
\end{equation}
when the $k$th partial derivative of $h$ exists.

\subsubsection{Exchangeable Pairs}
Applying Stein's method to a random vector~$W$ often  relies on creating a coupled random variable~$W'$ whose relationship to $W$ is used  to produce a bound on the distance between $W$ and its limiting distribution. An \textit{exchangeable pair} is a coupling $(W, W')$ such that $(W, W')$ is equal in distribution to $(W',W)$.  The main result we utilize, due to \citet{RR_exchangeable_pair_bound} (see  Theorem~\ref{RR_exchangeable_pair_bound}, below, requires a type of relaxed linearity condition
\begin{equation} \label{linearity condition multivariate}
\mathbb{E} [W' - W \mid W] = -\Lambda W + R
\end{equation}
between the exchangeable vector pair $(W,W')$, where  $\Lambda$ is an invertible matrix ~$R$ is an ``error'' term.
In all that follows, for a matrix $X$ let $\abs{X} := \sup \limits_{i, j} \abs{X_{ij}}$.

\begin{theorem}[\citet{RR_exchangeable_pair_bound}] \label{RR_exchangeable_pair_bound} 
	Assume that $(W, W')$ is an exchangeable pair of $\R^d$-valued random vectors such that 
	\begin{equation*}
	\mathbb{E} W = 0, \quad \mathbb{E} WW^\intercal = \Sigma
	\end{equation*}
	with $\Sigma \in \R^{d \times d}$ symmetric and positive definite. Suppose further that \eqref{linearity condition multivariate} is satisfied for an invertible matrix $\Lambda$ and a $\sigma(W)$-measurable random vector $R$. Then, if $Z$ has $d$-dimensional standard normal distribution, we have for every three times differentiable function $h$, 
	\begin{equation*}
	\abs{\mathbb{E} h(W) - \mathbb{E} h (\Sigma^{\frac{1}{2}} Z)} \leq \frac{\norm{h}_2}{4}A + \frac{\norm{h}_3}{12}B + \Big (\norm{h}_1 + \frac{1}{2} d \abs{\Sigma}^{1/2} \norm{h}_2 \Big ) C,
	\end{equation*}
	where, with $\lambda^{(i)} := \sum_{m = 1}^{d} \abs{(\Lambda^{-1})_{m,i}}$,
	\begin{align*}
	A &= \sum_{i, j = 1}^{d} \lambda^{(i)} \sqrt{\Var{\mathbb{E} \{(W_i' - W_i)(W_j' - W_j) \mid W \}}}, \\
	B &= \sum_{i, j, k = 1}^{d} \lambda^{(i)} \mathbb{E} \abs{(W_i' - W_i)(W_j' - W_j)(W_k' - W_k)}, \\
	C &= \sum_{i = 1}^{d} \lambda^{(i)} \sqrt{\Var{R_i}}
	\end{align*}
\end{theorem}

For a given random vector $W$, it may not be obvious which $W'$ is ``best'' to use in the theorem, and different choices will produce different bounds. From the  constants $A$, $B$, and $C$ one can generally see that a $W'$ is desirable if it produces an $R$ with small variance and $W'$ ``close'' to $W$. A common choice when $W$ is a sum of random vectors is to select uniformly at random one of the summands, replace it by an i.i.d.\ copy, and set $W'$ to be the resulting sum, which is our approach here.

\subsection{Asymptotic Normality of Maximum Likelihood Estimators}\label{sec:asymptotic_results_for_MLE}
For $Y_1, Y_2, \dots, Y_n \in \R^t$ independent but not necessarily identically distributed, let $f_i(y_i ; \theta)$ be the probability mass or density function of $Y_i$ with the parameter space $\Theta$ an open subset of $\R^d$. Suppose the observations $Y_1, \dots , Y_n$ are divided into $K$ groups with $n_k$ denoting the number of observations seen up to and including group $k$ for $k = 1, \dots, K$ with $n = n_K$. For example, $\{Y_1, Y_2, \dots, Y_{n_3}\}$ denotes all observations in the first 3 groups. We let $\hat{\theta}_k =  \hat{\theta}_k(Y_1, Y_2, \dots, Y_{n_k}) \in \R^d$ denote the MLE based on the observations in the first $k$ groups, and $G_k = \{n_{k-1}+1, \dots, n_k \}$ the set of indices of the observations in group $k$. The likelihood function at analysis $k$ is $L_k(\theta ; \bm{y}) = \prod_{i = 1}^{n_k} f_i(y_i | \theta)$ and the log-likelihood function is $\ell_{k}(\theta, \bm{y}) = \log (L_k(\theta ; \bm{y}))$. 

Below we state well-known regularity conditions that are sufficient for the asymptotic normality of the fixed-sample MLE, which we record as Theorem~\ref{thm:MLE asymtotic normality theorem}. These conditions are also sufficient for asymptotic normality in the group sequential setting, a result due to \cite{jt} which we record here as Theorem~\ref{thm:MLE sequential asymtotic normality theorem}.

\begin{enumerate}[label=(R\arabic*)]
\item $\hat{\theta}_n(Y) \overset{p}{\rightarrow} \theta_0$, as $n \rightarrow \infty$, where $\theta_0$ is the true parameter vector. \label{reg con 1}
	\item The score function, $S_i(Y_i ; \theta) = \nabla_{\theta} \log f_i(Y_i ; \theta) \in \R^{d\times 1}$ and the information matrix, $I_i(\theta) = \mathrm{Var}_{\theta} \left [ S_i(Y_i ; \theta) \right ] \in \R^{d \times d}, \quad i = 1, \dots, n$ exist almost surely with respect to the probability measure $\mathbb{P}$. \label{reg con 2}
	\item $I_i(\theta)$ is a continuous function of $\theta, \forall i = 1, 2, \dots, n$, almost surely with respect to $\mathbb{P}$ and is a measurable function of $Y_i$. \label{reg con 3}
	\item $\mathbb{E}_\theta [S_i(Y_i ; \theta)] = 0_d \in \R^{d \times 1}$ where $0_d$ is the $d$-column vector of all zeros. \label{reg con 4}
	\item $I_i(\theta) = \mathbb{E}_\theta [ [\nabla \log f_i(Y_i ; \theta)] [\nabla \log f_i(Y_i ; \theta)]^\intercal] = \mathbb{E}_{\theta} [-\nabla_{\theta} S_i^\intercal(Y_i ; \theta)]$. \label{reg con 5}
	\item \label{reg con 6} For 
	\begin{equation*}
	\bar{I}_n(k, \theta) = \frac{1}{n} \sum_{i=1}^{n_k} I_i(\theta) \; k = 1, \dots, K,
	\end{equation*}
	there exists a matrix  $\bar{I}(k, \theta) = \lim\limits_{n \rightarrow \infty} \bar{I}_n(k, \theta)$. In addition, $\bar{I}_n(k, \theta)$ and $\bar{I}(k, \theta)$ are symmetric and $\bar{I}(k, \theta)$ is positive definite for all $\theta$.
	\item For some $\delta > 0, \frac{\sum_{i} \mathbb{E}_{\theta_0} \abs{\bm{\lambda}^\intercal S_i(Y_i; \theta)}^{2+\delta}}{n^{\frac{2 +\delta}{2}}} \underset{n \rightarrow \infty}{\longrightarrow} 0$ for all $\bm{\lambda} \in \R^d$. \label{reg con 7}
	\item \sloppy With $\norm{\cdot}$ the ordinary Euclidean norm on $\R^d$, then for $i,j,u \in \{1,2, \dots, d\}$ there exists $\epsilon >0, C >0, \delta>0$ and random variables $B_{i,j,u}(Y_u)$ such that \label{reg con 8}
	\begin{enumerate}[label=(\roman*)]
		\item $\sup \left \{ \abs{\frac{\partial^2}{\partial \theta_i \partial \theta_j} \log (f_u(Y_u; t)) }\mid \norm{t-\theta_0} \leq \epsilon \right \} \leq B_{i,j,u}(Y_u)$,
		\item $\mathbb{E} \abs{B_{i,j,u}(Y_u)}^{1 + \delta} \leq C$.
	\end{enumerate}
\end{enumerate}
	
The classical result that MLEs are asymptotically normal is the following. See, for example, \citet{asymptotic_book}  for a proof, and \citet{hoadley1971} for a further discussion of the sufficient conditions.

\begin{theorem}\label{thm:MLE asymtotic normality theorem}
	Let $Y_1, Y_2, \dots, Y_n$ be independent random vectors with probability density (or mass) function $f_i(y_i; \theta)$, where $\theta \in \Theta \subset \R^d$. Assume that the MLE exists and is unique and that the regularity conditions \ref{reg con 1}-\ref{reg con 8} hold. Also let $Z \sim \mathcal{N}_d(0_d, 1_{d\times d})$ be the standard multivariate d-dimensional normal. Then,
	\begin{equation*}
	\sqrt{n}[\bar{I}_n(\theta_0)]^\frac{1}{2} \big( \hat{\theta}_n(Y) - \theta_0 \big) \overset{d}{\underset{n\rightarrow \infty}{\longrightarrow}} Z
	\end{equation*}
\end{theorem}

\begin{theorem}[\cite{jt}]\label{thm:MLE sequential asymtotic normality theorem}
	Suppose that observations $Y_i$ are independent with distributions~$f_i(y_i; \theta)$, where $\theta$ is $d$-dimensional, and that observations $Y_1, \dots, Y_{n_k}$ are available at analysis $k, k =1, \dots, K$. Let $n_k - n_{k-1} \rightarrow \infty $ such that $(n_k - n_{k-1})/n \rightarrow \gamma_k \in (0,1)$ for all $k = 1, \dots , K$. Furthermore let $\hat{\theta}_k$ denote the MLE of $\theta$ based on $Y_1, \dots, Y_{n_k}$. Suppose that the distributions $f_i$ are sufficiently regular so that \ref{reg con 1}-\ref{reg con 8} hold for each k. Also let $Z \sim \mathcal{N}_q(0_q, 1_{q\times q})$ and $\theta_0^K := [\theta_0^\intercal, \theta_0^\intercal, \dots, \theta_0^\intercal]^\intercal \in \R^q$. Then $\hat{\theta}^K = [\hat{\theta}_1^\intercal, \dots, \hat{\theta}_K^\intercal]^\intercal$ is asymptotically multivariate normal,
	\begin{equation*}
	\sqrt{n}[J_n]^{-\frac{1}{2}}(\hat{\theta}^K - \theta_0^K) \overset{d}{\underset{n\rightarrow \infty}{\longrightarrow}} Z
	\end{equation*}
	where,
	\begin{equation*}
	J_n = J_n(\theta_0) =\begin{bmatrix}
	\bar{I}_n^{-1}(1, \theta_0) & \bar{I}_n^{-1}(2, \theta_0) & \bar{I}_n^{-1}(3, \theta_0) & \dots & \bar{I}_n^{-1}(K, \theta_0) \\
	\bar{I}_n^{-1}(2, \theta_0) & \bar{I}_n^{-1}(2, \theta_0) & \bar{I}_n^{-1}(3, \theta_0) & &\\
	\bar{I}_n^{-1}(3, \theta_0) & \bar{I}_n^{-1}(3, \theta_0) & \bar{I}_n^{-1}(3, \theta_0) & &\\
	& \vdots 				  &						    & \ddots\\
	\bar{I}_n^{-1}(K, \theta_0) & \bar{I}_n^{-1}(K, \theta_0) & \bar{I}_n^{-1}(K,\theta_0) 	& \dots 						& \bar{I}_n^{-1}(K, \theta_0)					
	\end{bmatrix}.
	\end{equation*}
\end{theorem}

\section{A Bound to the Normal for Group Sequential Maximum Likelihood Estimators}\label{sec:main-result}
The asymptotic theory of group sequential MLEs in Theorem~\ref{thm:MLE sequential asymtotic normality theorem} due to  \citet{jt} guarantee asymptotic normality under the appropriate conditions.  In this section we present results that give an error bound for the normal approximation under the same conditions.  In Section~\ref{subsec:main.result} we present the main result in Theorem~\ref{my_result_1}, whose proof is delayed until Appendix~\ref{sec:mle_non_asymptotic_proof} but after the theorem we outline an argument showing that the bound is asymptotically $\mathcal{O}(n^{-1/2})$. In Section~\ref{section proposed research} we discuss a result of \citet{gaunt_new_bounds_on_solutions_moment_conditions} allowing a relaxation on the number of needed derivatives of the test function~$h$, and apply this to achieve a similar relaxation to Theorem~\ref{RR_exchangeable_pair_bound} in Theorem~\ref{thm:my-result-1-improved}, which may be of independent interest.  Then in Theorem~\ref{thm:my-result-1-improved} we pass this relaxation along to Theorem~\ref{my_result_1}.

\subsection{Main Result}\label{subsec:main.result}

For ease of presentation we introduce the following additional notation. For $\theta_0^K := [\theta_0^\intercal, \theta_0^\intercal, \dots, \theta_0^\intercal]^\intercal \in \R^q$, let $Q_i := [\hat{\theta}^K - \theta_0^K]_i$ and $Q_{(m)} :=\max_{i \in \{1, \dots, q\}} Q_i$. Let $Y_i'$ be an independent copy of $Y$, 
$$\xi_{ij} = [n^{-\frac{1}{2}} S_i(Y_i, \theta_0)]_j, \qmq{and} \xi_{ij}' = [n^{-\frac{1}{2}} S_i(Y_i', \theta_0)]_j.$$	
Define 
\begin{align*}
A &= \begin{bmatrix}
1_{d \times d} & 0 & 0 & \dots & 0 \\
1_{d \times d} & 1_{d \times d} & 0 & &0 \\
1_{d \times d} & 1_{d \times d} & 1_{d \times d} & & 0 \\
& \vdots & & \ddots \\
1_{d \times d} & 1_{d \times d} & 1_{d \times d} & \dots & 1_{d \times d}
\end{bmatrix} \in \R^{q \times q},\\
\tilde{J_n} = \tilde{J_n}(\theta) &= \begin{bmatrix}
\bar{I}_n(1, \theta) & \bar{I}_n(1, \theta) & \bar{I}_n(1, \theta) & & \bar{I}_n(1, \theta) \\
\bar{I}_n(1, \theta) & \bar{I}_n(2, \theta) & \bar{I}_n(2, \theta) &\dots &\bar{I}_n(2, \theta) \\
\bar{I}_n(1, \theta) & \bar{I}_n(2, \theta) & \bar{I}_n(3, \theta) & & \bar{I}_n(3, \theta)\\
& \vdots 				  &						    & \ddots\\
\bar{I}_n(1, \theta) & \bar{I}_n(2, \theta) & \bar{I}_n(3,\theta) 	& \dots 						& \bar{I}_n(K, \theta)					
\end{bmatrix} \in \R^{q \times q}
\end{align*}
where $1_{d \times d}$ is the $d$-dimensional identity matrix. We denote the $n$th row and $m$th column of a matrix $W$ as
\begin{align*}
n\text{th row: }& W_{n *} =W_{n, *} \\
m\text{th columns: }& W_{* m} =W_{*, m}.
\end{align*} 
When $W$ is defined as a block matrix, we let $W_{[i] [j]}$ be the sub matrix in `block row' $i$ and `block column' $j$. Similarly when $W \in \R^q$ is defined as a block vector, we let $W_{[i]} \in \R^d$ be the $i$th sub vector. We denote the $n$th block row and $m$th block column as 
\begin{align*}
n\text{th block row: }& W_{[n] [*]} = W_{[n], [*]} \\
m\text{th block columns: }& W_{[*] [m]} = W_{[*], [m]}.
\end{align*}
We have the notation in place to state the main result of the paper.
\begin{theorem}\label{my_result_1}
	Let $Y_1, Y_2, \dots, Y_n$ be independent but possibly non-identically distributed $\R^t$-valued random vectors with probability density (or mass) functions $f_i(y_i | \theta)$, for which the parameter space $\Theta$ is an open subset of $\R^d$. Suppose the observations $Y_1, Y_2, \dots, Y_{n_k}$ are available at analysis $k= 1, \dots, K$. Assume that the \textit{MLE} $\hat{\theta}_k$ exists and is unique and that conditions \ref{reg con 1}-\ref{reg con 6} are satisfied at each analysis $k$. In addition, assume that for any $\theta_0$ there exists $0<\epsilon = \epsilon(\theta_0)$ and functions $M^k_{iuj}(y)$, $\forall i, u, j=1, 2, \dots, d$ such that $\abs{ \frac{\partial^3}{\partial \theta_i \partial \theta_u \partial \theta_j} \ell_k(\theta, y) } \leq M^k_{iuj}(y)$ for all $\theta \in \Theta$ with $\abs{\theta_j - \theta_{0, j}} < \epsilon \; \forall j =1, 2, \dots, d$. Also, assume that $\mathbb{E} [\big (\frac{\partial^3}{\partial \theta_i\partial \theta_u \partial \theta_j} M^k_{iuj}(Y) \big )^2 \mid \abs{Q_{(m)}} < \epsilon ] < \infty$ for all $k=1, \dots, K$. Let $\{Y'_i, i = 1, 2, \dots, n \}$ be an independent copy of $\{Y_i, i = 1, 2, \dots, n \}$. For $Z \sim N_q(0_q, 1_{q \times q}), h \in \mathcal{H}$, where $\mathcal{H}$ is as in \eqref{H}, it holds that
	\begin{multline}\label{my_result_1_statement}
	\abs{\mathbb{E} [h(\sqrt{n} J^{-\frac{1}{2}}_n (\hat{\theta}^K - \theta_0^K))] - \mathbb{E} [ h(Z)]} \leq \frac{\norm{h}_1}{\sqrt{n}} K_1(\theta_0) \\ + \frac{ q^2 c^2 \norm{h}_2}{4} K_2(\theta_0) + \frac{  q^3 c^3 \norm{h}_3}{12} K_3(\theta_0) + \frac{2\norm{h}}{\epsilon^2} \mathbb{E} \big [\sum_{j = 1}^{q}Q_{j}^2 \big],
	\end{multline}
	where
	\begin{align*}
	&K_1(\theta_0) = \sum_{k_1 = 1}^{K} \sum_{k_2 = k_1}^{\min\{k_1 + 1, K\}} \sum_{l = 1}^{d} \sum_{j = 1}^{d} \abs{\bar{I}_n^{-\frac{1}{2}}(G_{k_2}; \theta_0)_{lj}} \Bigg \{ \Bigg (  \\
	&\sum_{i = 1}^{d} \sqrt{ \mathbb{E} [[\hat{\theta}_{[k_1]} - \theta_0]_i]^2 \mathbb{E} \Big (\frac{\partial^2}{\partial \theta_i \partial \theta_l} \ell_{k_1}(\theta_0, Y) + n\bar{I}_n(k_1, \theta_0)_{li}\Big)^2 }   \Bigg )  \\
	&+ \frac{1}{2} \sum_{i = 1}^{d}\sum_{u = 1}^{d} \big (\mathbb{E} \big \{([\hat{\theta}_{[k_1]} - \theta_0]_i)^2 ([\hat{\theta}_{[k_1]} - \theta_0]_u)^2 \big \} \big )^{\frac{1}{2}} (\mathbb{E} [ \big (M^{k_1}_{iul}(Y) \big)^2 \mid \abs{Q_{(m)}} < \epsilon ] )^\frac{1}{2} \Bigg\} \\
	&K_2(\theta_0) = \sum_{k = 1}^{K} \Big\{ \sum_{j = 1}^d \Big [ \sum_{i \in G_k} \Var{\xi_{ij}^2} \Big ]^{\frac{1}{2}} + 2 \sum_{i < j} \Big [ \sum_{v \in G_k} \mathrm{Var} [ \xi_{vi} \xi_{vj}] \Big ]^ {\frac{1}{2}} \Big \} \\
	&K_3(\theta_0) = \sum_{i = 1}^{n} \mathbb{E} \Big [ \sum_{j = 1}^d \abs{(\xi'_{ij} - \xi_{ij})}  \Big ]^3\\
	&c = \max_{k \in \{1, \dots, K\}} \abs{\bar{I}^{-1/2}_n(G_k;\theta_0)}.
	\end{align*}
\end{theorem}

The theorem is proved in Appendix~\ref{sec:mle_non_asymptotic_proof}.

Although the bound in the theorem appears complex, the fact that it is in general of optimal order $\mathcal{O}(n^{-1/2})$ is not hard to see. Assume that $[\bar{I}_n(k, \theta_0)]_j = \mathcal{O}(1)$. By Theorem~\ref{thm:MLE sequential asymtotic normality theorem},
\begin{equation*}
\sqrt{n}\mathbb{E}[\hat{\theta}_{k} - \theta_0] \rightarrow 0_d.
\end{equation*}
Thus $\mathbb{E}[\hat{\theta}_{k} - \theta_0]_j =o(1/\sqrt{n})$. Again from Theorem \ref{thm:MLE sequential asymtotic normality theorem},
\begin{equation*}
n \bar{I}^\frac{1}{2}_n(k, \theta_0)\mathrm{Cov}[\hat{\theta}_k]\bar{I}^\frac{1}{2}_n(k, \theta_0) \rightarrow 1_{d \times d}.
\end{equation*}
Using $[\bar{I}_n(k, \theta_0)]_j = \mathcal{O}(1)$, we see that $n\Var{[\hat{\theta_k}]_j} \rightarrow 1$ and thus $\Var{[\hat{\theta_k}]_j} = \mathcal{O}(1/n)$. It follows that
\begin{equation}\label{big_O1}
\mathbb{E} ([\hat{\theta}_k - \theta_0]^2_j) = \Var{[\hat{\theta_k}]_j} + (\mathbb{E}[\hat{\theta}_{k} - \theta_0]_j)^2 = \mathcal{O} (n^{-1} ).
\end{equation}
	
From \ref{reg con 2} and \ref{reg con 6} and independence of $Y_1, \dots, Y_{n_k}$ it can be seen that,
\begin{equation}\label{big_O2}
\mathbb{E} \Big (\frac{\partial^2}{\partial \theta_i \partial \theta_l} \ell_{k}(\theta_0, Y) + n\bar{I}_n(k, \theta_0)_{li}\Big)^2 = \sum_{j = 1}^{n_k} \Var{\frac{\partial^2}{\partial \theta_i \partial \theta_l} \log(f_j(Y_j \mid \theta_0))}
\end{equation}
showing that the left hand side of the above equation is $\mathcal{O}(n)$. It follows from \eqref{big_O1}, \eqref{big_O2}, and $[\bar{I}_n(k, \theta_0)]_j = \mathcal{O}(1)$ that $K_1(\theta_0)= \mathcal{O}(1)$ and $K_3(\theta_0)= \mathcal{O}(n^{-\frac{1}{2}})$. That $K_2(\theta_0)= \mathcal{O}(n^{-\frac{1}{2}})$ depends on the fourth moment of $\xi_{ij}$ being $\mathcal{O}\left(\frac{1}{n^2}\right)$. In general we have reason to believe this is true but have not proved it. Below we show for exponential families that $K_2$ is indeed $\mathcal{O}(n^{-1/2})$. When all this holds, the right hand side of \eqref{my_result_1_statement} is indeed $\mathcal{O}(n^{-1/2})$. For extensive details of the order of the bound one may refer to the analysis in Section~\ref{sec:mle-exponential-fam-bound} for the simpler case of exponential families.

\subsection{Extensions of Theorem~\ref{my_result_1}} \label{section proposed research}

Theorem~\ref{RR_exchangeable_pair_bound}, which is the key tool in proving our Theorem~\ref{my_result_1}, relies on \eqref{properties of the solution multivariate} to bound the derivatives of the multivariate Stein solution.  \citet[][Proposition~2.1]{gaunt_new_bounds_on_solutions_moment_conditions} found new bounds on the derivatives of the Stein solution that require one fewer derivative of the function~$h$ in \eqref{properties of the solution multivariate}. In Theorem~\ref{thm:RR-improved}, we use this result to relax the conditions of Theorem~\ref{RR_exchangeable_pair_bound} by requiring that $h$ be two times differentiable instead of three. The price paid is an increase in the order of the bound with respect to $d$ by a factor of $d^{1/2}$.
\begin{theorem}\label{thm:RR-improved}
	Assume that $(W, W')$ is an exchangeable pair of $\R^d$-valued random vectors such that 
	\begin{equation*}
	\mathbb{E} W = 0, \quad \mathbb{E} WW^\intercal = \Sigma
	\end{equation*}
	with $\Sigma \in \R^{d \times d}$ symmetric and positive definite. Suppose further that \eqref{linearity condition multivariate} is satisfied for an invertible matrix $\Lambda$ and a $\sigma(W)$-measurable random vector $R$. Then, if $Z$ has $d$-dimensional standard normal distribution, we have for every two times differentiable function $h$, 
	\begin{align*}
	\abs{\mathbb{E} h(W) - \mathbb{E} h (\Sigma^{\frac{1}{2}} Z)} &\leq 
	d^{1/2}\abs{\Sigma^{-\frac{1}{2}}}\Bigg (\frac{\norm{h}_1}{\sqrt{\pi}}A + \frac{\norm{h}_2 \sqrt{2\pi}}{8}B \\ 
	&+ \Big (\sqrt{\frac{\pi}{2}}\norm{h-\mathbb{E}h(\Sigma^{\frac{1}{2}} Z)} + \frac{2d}{\sqrt{\pi}} \abs{\Sigma}^{1/2} \norm{h} \Big ) C \Bigg),
	\end{align*}
	where, with $\lambda^{(i)} := \sum_{m = 1}^{d} \abs{(\Lambda^{-1})_{m,i}}$,
	\begin{align*}
	A &= \sum_{i, j = 1}^{d} \lambda^{(i)} \sqrt{\Var{\mathbb{E} \{(W_i' - W_i)(W_j' - W_j) \mid W \}}}, \\
	B &= \sum_{i, j, k = 1}^{d} \lambda^{(i)} \mathbb{E} \abs{(W_i' - W_i)(W_j' - W_j)(W_k' - W_k)}, \\
	C &= \sum_{i = 1}^{d} \lambda^{(i)} \sqrt{\Var{R_i}}
	\end{align*}
\end{theorem}
\begin{proof}
	The proof is similar to the proof of Theorem~\ref{RR_exchangeable_pair_bound}, with  \citet[][Proposition~2.1]{gaunt_new_bounds_on_solutions_moment_conditions} used in place of \eqref{properties of the solution multivariate}. The details are this omitted. 
\end{proof}

We now pass along the improvement of this bound to enhance our main result, Theorem \ref{my_result_1}. 

\begin{theorem}\label{thm:my-result-1-improved}
	Let $Y_1, Y_2, \dots, Y_n$ be independent non-identically distributed $\R^t$-valued random vectors with probability density (or mass) functions $f_i(y_i | \theta)$, for which the parameter space $\Theta$ is an open subset of $\R^d$. Suppose the observations $Y_1, Y_2, \dots, Y_{n_k}$ are available at analysis $k= 1, \dots, K$. Assume that the \textit{MLE} $\hat{\theta}_k$ exists and is unique and that conditions \ref{reg con 1}-\ref{reg con 6} are satisfied at each analysis $k$. In addition, assume that for any $\theta_0$ there exists $0<\epsilon = \epsilon(\theta_0)$ and functions $M^k_{iuj}(y)$, $\forall i, u, j=1, 2, \dots, d$ such that $\abs{ \frac{\partial^3}{\partial \theta_i \partial \theta_u \partial \theta_j} \ell_k(\theta, y) } \leq M^k_{iuj}(y)$ for all $\theta \in \Theta$ with $\abs{\theta_j - \theta_{0, j}} < \epsilon \; \forall j =1, 2, \dots, d$. Also, assume that $\mathbb{E} [\big (\frac{\partial^3}{\partial \theta_i\partial \theta_u \partial \theta_j} M^k_{iuj}(Y) \big )^2 \mid \abs{Q_{(m)}} < \epsilon ] < \infty$ for all $k=1, \dots, K$. Let $\{Y'_i, i = 1, 2, \dots, n \}$ be an independent copy of $\{Y_i, i = 1, 2, \dots, n \}$. For $Z \sim N_q(0_q, 1_{q \times q}), h \in \mathcal{H}$, where $\mathcal{H}$ is the class of all bounded functions with bounded first and second order derivatives, it holds that
	\begin{multline*}
	\abs{\mathbb{E} [h(\sqrt{n} J^{-\frac{1}{2}}_n (\hat{\theta}^K - \theta_0^K))] - \mathbb{E} [ h(Z)]} \leq \frac{\norm{h}_1}{\sqrt{n}} K_1(\theta_0) \\ + \frac{ q^{3/2} c^2 \norm{h}_1}{\sqrt{\pi}} K_2(\theta_0) + \frac{\sqrt{2 \pi}  q^{5/2} c^3 \norm{h}_2}{8} K_3(\theta_0) + \frac{2\norm{h}}{\epsilon^2} \mathbb{E} \big [\sum_{j = 1}^{q}Q_{j}^2 \big],
	\end{multline*}
	where,
	\begin{align*}
	&K_1(\theta_0) = \sum_{k_1 = 1}^{K} \sum_{k_2 = k_1}^{\min\{k_1 + 1, K\}} \sum_{l = 1}^{d} \sum_{j = 1}^{d} \abs{\bar{I}_n^{-\frac{1}{2}}(G_{k_2}; \theta_0)_{lj}} \Bigg \{ \Bigg ( \\
	&\sum_{i = 1}^{d} \sqrt{ \mathbb{E} [[\hat{\theta}_{[k_1]} - \theta_0]_i]^2 \mathbb{E} \Big (\frac{\partial^2}{\partial \theta_i \partial \theta_l} \ell_{k_1}(\theta_0, Y) + n\bar{I}_n(k_1, \theta_0)_{li}\Big)^2 }   \Bigg ) \\
	&+ \frac{1}{2} \sum_{i = 1}^{d}\sum_{u = 1}^{d} \big (\mathbb{E} \big \{([\hat{\theta}_{[k_1]} - \theta_0]_i)^2 ([\hat{\theta}_{[k_1]} - \theta_0]_u)^2 \big \} \big )^{\frac{1}{2}} (\mathbb{E} [ \big (M^{k_1}_{iul}(Y) \big)^2 \mid \abs{Q_{(m)}} < \epsilon ] )^\frac{1}{2} \Bigg\} \\
	&K_2(\theta_0) = \sum_{k = 1}^{K} \Big\{ \sum_{j = 1}^d \Big [ \sum_{i \in G_k} \Var{\xi_{ij}^2} \Big ]^{\frac{1}{2}} + 2 \sum_{i < j} \Big [ \sum_{v \in G_k} \mathrm{Var} [ \xi_{vi} \xi_{vj}] \Big ]^ {\frac{1}{2}} \Big \} \\
	&K_3(\theta_0) = \sum_{i = 1}^{n} \mathbb{E} \Big [ \sum_{j = 1}^d \abs{(\xi'_{ij} - \xi_{ij})}  \Big ]^3\\
	&c = \max_{k \in \{1, \dots, K\}} \abs{\bar{I}^{-1/2}_n(G_k;\theta_0)}.
	\end{align*}
\end{theorem}
\begin{proof}  The proof of Theorem~\ref{my_result_1} (see Appendix~\ref{sec:mle_non_asymptotic_proof}) only needs to be augmented by using Theorem~\ref{thm:RR-improved} in place of Theorem~\ref{RR_exchangeable_pair_bound} and noting that by \eqref{eq:Sigma_to_the_neg_1/2}, $\abs{\Sigma^{-1/2}} = c$.  With this adjustment, the rest of the proof is similar and the details are thus omitted.
\end{proof}

\cite{fang_wasserstein_multivariate_near_optimal} use Malliavin calculus techniques along with an exchangeable pair approach to attain a near-optimal error bound on the Wasserstein distance for multivariate approximations. Using this, a further improvement of Theorem~\ref{thm:my-result-1-improved} is possible which removes the requirement that $h$ have bounded second order derivatives at the expense of sacrificing optimal order convergence.

\section{Application to Observations from an Exponential Family}\label{sec:exp-fam}
In this section we specialize Theorem~\ref{my_result_1} to the case where observations are i.i.d.\ from an exponential family. We then show that this bound is of optimal order $\mathcal{O}(n^{-1/2})$ and calculate the bound explicitly in the case of lifetime data from an exponential distribution.

\subsection{Notation and Setup}
We slightly modify some of our notation to be more in line with the exponential family literature. We say that the distribution of $Y$ is a canonical multi-parameter exponential family distribution if for $\bm{\eta}\in \R^d$ the density of $Y$ is,
\begin{equation}\label{eq:exponential-family}
	f(y; \bm{\eta}) = \exp \left \{ \sum_{j=1}^{d}\eta_j T_j(y) - A(\bm{\eta}) + S(y) \right\} \mathbbm{1}_{\{y \in B\}}
\end{equation}
where $B$ is the support of $y$ that does not depend on $\bm{\eta}$, and $T(y) = [T_1(y), \dots, T_d(y)]^\intercal$ is the natural sufficient statistic. Here, the natural parameter $\bm{\eta}$ plays the role of the parameter of interest~$\theta$ above. We will use $\eta$ and $\bm{\eta}$ interchangeably as long as the multidimensionality of $\eta$ is clear from the context.  The cumulant function~$A(\eta)$ satisfies
\begin{equation}\label{eq:exponential-family-expectation-of-suff-statistic}
	\frac{\partial}{\partial \eta_i}A(\eta) = \Esub{\eta}{T_i(Y)}, \quad \Varsub{\eta}{T(Y)} = H(A(\eta)) = \nabla_\eta^2 A(\eta).
\end{equation}
To ease the notational burden we use the following simplified notation for the mixed moments of the sufficient statistics:
\begin{gather*}
\mu_i = \mu_i(\eta) = \Esub{\eta}{T_i(Y)} \\
\mu_{ij} = \mu_{ij}(\eta) = \Esub{\eta}{T_i(Y)T_j(Y)}\\
\mu_{ijk} = \mu_{ijk}(\eta) = \Esub{\eta}{T_i(Y)T_j(Y) T_k(Y)}.
\end{gather*}
In Section~\ref{sec:mle-exponential-fam-bound} we will also make use of the third partial derivative of $A(\eta)$ which appear in Theorem~\ref{my_result_1}. We have
\begin{multline}
\label{eq:cumulant-3rd-derivative}
	\frac{\partial^3}{\partial \eta_i \partial \eta_j \partial \eta_k} A(\eta) = \int \exp \{S(y)\} T_i(y) \bigg [ \\ 
	\left(T_j(y) - \frac{\partial}{\partial \eta_j} A(\eta) \right) \left(T_k(y) - \frac{\partial}{\partial \eta_k} A(\eta) \right) a(y, \eta)  - a(y, \eta) \frac{\partial^2}{\partial \eta_j \partial \eta_k} A(\eta) \bigg]dy  \\
	= \mu_{ijk} - \mu_{ij}\mu_{k} - \mu_{ik}\mu_j  + \mu_i\mu_j \mu_{k} - \Cov{T_j(Y)}{T_k(Y)} \mu_i \\
	= \mu_{ijk} - \mu_{ij}\mu_{k} - \mu_{ik}\mu_j -\mu_{jk}\mu_i + 2\mu_i\mu_j \mu_{k}.
\end{multline}
Define
\begin{equation}\label{eq:exp-fam-mu-epsilon}
	\mu_{iul}^\epsilon = \mu_{iul}^\epsilon(\eta_0) = \max_{\{\eta: \abs{\eta_j - \eta_{0,j}} < \epsilon, \; \forall j \in \{1, \dots, d\}\}} \abs{\mu_{ilk} - \mu_{il}\mu_{k} - \mu_{ik}\mu_l -\mu_{lk}\mu_i + 2\mu_i\mu_l \mu_{k}}
\end{equation}

The joint cumulant generating function of $W=(T_1(Y), \dots, T_d(Y))^\intercal$ is then
\begin{equation*}
	K(s) = \log \E{e^{s^\intercal W}}= A(s+\eta) -A(\eta).
\end{equation*}
Via this expression, the cumulants can be found by taking derivatives of $A(\eta)$, as follows:
\begin{equation*}
	\nabla_{s^n} K(s) \big \rvert_{s=0} = \nabla_{s^n} A(s+ \eta) \big \rvert_{s=0} = \nabla_{\eta^n} A(\eta).
\end{equation*}
Thus partial derivatives of $A(\eta)$ yield joint cummulants of $[T_1(Y), \dots, T_d(Y)]$ which in turn are polynomial functions of the moments of $[T_1(Y), \dots, T_d(Y)]$.

We now turn to MLEs of exponential families, which we denote by $\hat{\eta}$. The score function is
$$\nabla_\eta \ell(y^n; \eta) = \nabla_\eta \sum_{i =1}^n \left ( \sum_{j = 1}^d \eta_j T_j(y_i) - A(\eta) + S(y_i) \right) 
	= \sum_{i =1}^n T(y_i) - n\nabla_\eta A(\eta)$$
amd setting this equal to zero and using \eqref{eq:exponential-family-expectation-of-suff-statistic} yields
\begin{equation*}
	\frac{1}{n}\sum_{i=1}^n T(y_i) = \nabla_\eta A(\eta) \big \rvert_{\eta = \hat{\eta}} = \Esub{\hat{\eta}}{T(Y)}.
\end{equation*}
Since the log-likelihood of an exponential family is strictly concave, if the MLE exists it is a global maximum. The mean function is
\begin{equation*}
	\tau(\eta) = [\tau_1(\eta), \dots, \tau_d(\eta)]^\intercal = [\Esub{\eta}{T_1(Y)}, \dots, \Esub{\eta}{T_d(Y)}]^\intercal,
\end{equation*}
we write compactly as
\begin{equation}\label{eq:exponential-fam-mle-special-form}
	\tau(\hat{\eta}) = \frac{1}{n}\sum_{i=1}^{n}T(y_i)\qmq{or, equivalently,}\hat{\eta} = \tau^{-1} \left(\frac{1}{n}\sum_{i=1}^{n}T(y_i) \right ).
\end{equation}

We note that \eqref{eq:exponential-fam-mle-special-form} matches the form~\eqref{eq:mle-special-form} and thus the bounds found in \citet{delta} and \citet{mvn_delta} can be used on the distance to normality for MLEs of exponential families. This fact was noted in \cite{delta} for single parameter exponential families but the generalization to multi-parameter exponential families was not explicitly addressed in \cite{mvn_delta}.


\subsection{Bound to Normality for MLEs of Exponential Families}\label{sec:mle-exponential-fam-bound}

With the notation in place we are ready to state the specialization of Theorem~\ref{my_result_1} to exponential families in Corollary~\ref{cor:exp-fam}. The bound~\eqref{ex:exp-fam-statement} that results can be written completely in terms of mixed moments up to the third order of the sufficient statistics $T_1(Y), \dots, T_d(Y)$, the true value of the natural parameter $\eta_0$, and the sample. Additional knowledge of the distribution, in particular the cumulant function $A(\eta)$ and $S(Y)$, is not required. Although the bound in Corollary~\ref{cor:exp-fam} inherits the optimal order of $n^{-1/2}$ from Theorem~\ref{my_result_1}, a direct analysis of the bound's order is more straightforward due to the observations being i.i.d., and we carry this out following the corollary in Section~\ref{sec:exp.fam.order}.

\begin{corollary}\label{cor:exp-fam}
	Let $Y_1, Y_2, \dots, Y_n$ be i.i.d with probability density (or mass) function $f(y | \eta)$, for which the parameter space is an open subset of $\R^d$. Suppose $f$ is an exponential family defined by \eqref{eq:exponential-family} and the observations $Y_1, Y_2, \dots, Y_{n_k}$ are available at analysis $k = 1, \dots , K$.  Assume that the \textit{MLE} $\hat{\eta}_k$ exists and that conditions \ref{reg con 1}-\ref{reg con 6} are satisfied at each analysis $k$. Let $\{Y'_i, i = 1, 2, \dots, n \}$ be an independent copy of $\{Y_i, i = 1, 2, \dots, n \}$. For $Z \sim N_q(0_q, 1_{q \times q}), h \in \mathcal{H}$, where $\mathcal{H}$ is as in \eqref{H} and $\mu^{\epsilon}_{iul}$ defined by \eqref{eq:exp-fam-mu-epsilon}, it holds that
	\begin{multline}\label{ex:exp-fam-statement}
	\abs{\mathbb{E} [h(\sqrt{n} J^{-\frac{1}{2}}_n (\hat{\eta}^K - \eta_0^K))] - \mathbb{E} [ h(Z)]} \leq \frac{\norm{h}_1}{\sqrt{n}} K_1(\eta_0) \\ + \frac{ q^2 c^2 \norm{h}_2}{4} K_2(\eta_0) + \frac{  q^3 c^3 \norm{h}_3}{12} K_3(\eta_0) + \frac{2\norm{h}}{\epsilon^2} \mathbb{E} \big [\sum_{j = 1}^{q}Q_{j}^2 \big],
	\end{multline}
	where,
	\begin{align*}
	&K_1(\eta_0) = \frac{1}{2} \sum_{k_1 = 1}^{K} \sum_{k_2 = k_1}^{\min\{k_1 + 1, K\}} n_{k_1} \left (\frac{\abs{G_{k_2}}}{n} \right)^{-\frac{1}{2}}\sum_{l = 1}^{d} \sum_{j = 1}^{d} \mathrm{Var}^{-\frac{1}{2}}\left[{T(Y)}\right]_{lj} \Bigg \{ \\
	& \sum_{i = 1}^{d}\sum_{u = 1}^{d} \mu^{\epsilon}_{iul} \Bigg (\mathbb{E} \Bigg[\left (\tau^{-1}\left (\frac{1}{n_{k_1}}\sum_{s=1}^{n_{k_1}}T(y_s) \right )_i - \eta_{0,i}\right )^2 \times \\ &\left (\tau^{-1}\left (\frac{1}{n_{k_1}}\sum_{s=1}^{n_{k_1}}T(y_s) \right )_u - \eta_{0,u}\right )^2 \Bigg] \Bigg )^{\frac{1}{2}} \Bigg\} \\
	&K_2(\eta) =\frac{1}{\sqrt{n}}\sum_{k = 1}^{K} \left(\frac{\abs{G_k}}{n}\right)^{\frac{1}{2}}\Big\{ \sum_{j = 1}^d \Big [ \Var{(T_j(Y) - \mu_j(\eta_0))^2} \Big ]^{\frac{1}{2}} + \\
	&2 \sum_{i < j} \Big [ \Var{(T_i(Y) - \mu_i(\eta_0))(T_j(Y) - \mu_j(\eta_0))} \Big ]^ {\frac{1}{2}} \Big \} \\
	&K_3(\eta_0) = \frac{1}{\sqrt{n}} \mathbb{E} \Big [ \sum_{j = 1}^d \abs{T_j(Y') -T_j(Y)}  \Big ]^3\\
	&c = \abs{\mathrm{Var}^{-\frac{1}{2}}\left[{T(Y)}\right]}\max_{k \in \{1, \dots, K\}} \left (\frac{\abs{G_{k}}}{n} \right)^{-\frac{1}{2}}.
	\end{align*}
\end{corollary}

\subsection{The bound in Corollary~\ref{cor:exp-fam} is of order $\mathcal{O}(n^{-1/2})$}\label{sec:exp.fam.order}
In Section~\ref{sec:main-result} we analyzed the order of the bound in Theorem~\ref{my_result_1} and found it is of order $n^{-1/2}$ given some mild assumptions. That it is still $\mathcal{O}(n^{-1/2})$ when specialized to the exponential family case is no surprise, and can be seen more directly here.  Letting $\abs{G_k}/n \To \gamma_k\in(0,1)$ as in Theorem~\ref{thm:MLE sequential asymtotic normality theorem}, the term~$c$ in the bound satisfies
 \begin{equation*}
	c\To \abs{\mathrm{Var}^{-\frac{1}{2}}\left[{T(Y)}\right]}\max_{k \in \{1, \dots, K\}} \gamma_k^{-\frac{1}{2}}.
\end{equation*}
and thus $c=\mathcal{O}(1)$. For the same reason, $K_2(\eta)=\mathcal{O}(n^{-1/2})$. The function $K_3(\eta)$ is exactly equal to a constant times $n^{-1/2}$ and so is $\mathcal{O}(n^{-1/2})$. The last term in the bound,
\begin{equation*}
	\frac{2\norm{h}}{\epsilon^2}\E{\sum_{j = 1}^{q}Q_j^2}
\end{equation*}
depends on the order of $$\E{Q_j^2} = \E{(\hat{\eta}_j - \eta_{0,j})^2} = \Var{\hat{\eta}_j} + (\E{Q_j})^2 = \mathcal{O}\left(n^{-1}\right).$$ This leaves only $K_1(\eta_0)$ to be considered. For the entire bound to be $\mathcal{O}(n^{-1/2})$, $K_1(\eta)$ should be $\mathcal{O}(1)$ since $K_1(\eta)$ is multiplied by $n^{-1/2}$. This is the case if and only if
\begin{equation*}
	\Bigg (\mathbb{E} \Bigg[\left (\tau^{-1}\left (\frac{1}{n_{k_1}}\sum_{s=1}^{n_{k_1}}T(y_s) \right )_i - \eta_{0,i}\right )^2 \left (\tau^{-1}\left (\frac{1}{n_{k_1}}\sum_{s=1}^{n_{k_1}}T(y_s) \right )_u - \eta_{0,u}\right )^2 \Bigg] \Bigg )^{\frac{1}{2}}
\end{equation*}
is $\mathcal{O}(n^{-1})$. By the Cauchy-Schwartz inequality, the above is bounded by
\begin{equation*}
	(\E{(\hat{\eta}_i - \eta_{0,i})^4}\E{(\hat{\eta}_u - \eta_{0,u})^4})^\frac{1}{4}.
\end{equation*}
Thus the desired order follows from  the fact that the fourth central moment of a component of the MLE is $\mathcal{O}\left(n^{-2}\right)$; see \citet[equation $2$]{jose_jose_2001}  for the univariate exponential case and \cite{peers_iqbal1985}  for the general multivariate case.

The choice of $\epsilon$ in the bound should be handled on a case-by-case basis and can be chosen to minimize the sum of the $K_1$ term and the final $\frac{2\norm{h}}{\epsilon^2}$ term. Notice that smaller $\epsilon$ increases the final term while decreases $\mu_{iul}^\epsilon$ in $K_2$.

\subsection{Example: The Exponential Distribution}
In this section we apply Corollary~\ref{cor:exp-fam} to the case where observations are i.i.d.\ from an exponential distribution. We use the natural parameterization 
\begin{equation*}
	T(y) = -y, \quad A(\eta) = - \log(\eta), \quad S(y) = 0, \quad B= [0, \infty).
\end{equation*}
We record the application of Corollary~\ref{cor:exp-fam} to this distribution in the following corollary, in which we utilize that $\hat{\eta}_{n_k} \sim \text{Inv-Gamma}(n_k, n_k\lambda)$ and $\mu^\epsilon_{iul} = 2/(\eta_0-\epsilon)^3$. The remaining  calculations as they are similar to, and simpler than, those for Corollary~\ref{cor:exp-fam} ans we thus omit them.

\begin{corollary}
	Let $Y_1, Y_2, \dots, Y_n$ be independent such that $Y_i \sim \mathrm{Exp}(\eta_0)$ with $\E{Y_i} = \frac{1}{\eta_0}$. Suppose the observations $Y_1, Y_2, \dots, Y_{n_k}$ are available at analysis $k= 1, \dots, K$. For $Z \sim N_q(0_q, 1_{q \times q}), h \in \mathcal{H}$, where $\mathcal{H}$ is as in \eqref{H}, it holds that
	\begin{multline*}
	\abs{\mathbb{E} [h(\sqrt{n} J^{-\frac{1}{2}}_n (\hat{\eta}^K - \eta_0^K))] - \mathbb{E} [ h(Z)]} \leq \frac{\norm{h}_1}{\sqrt{n}} K_1(\eta_0) + \frac{2K^2 c^2 \norm{h}_2}{\eta_0^2\sqrt{n}} K_2(\eta_0) \\ + \frac{2\eta_0^2\norm{h}}{\epsilon^2} \sum_{k = 1}^{K} \left(\frac{1}{n_k}\right)\frac{n_k^2+2n_k}{(n_k-1)(n_k-2)},
	\end{multline*}
	where,
	\begin{align*}
	&K_1(\eta_0) = \frac{\eta_0^3\sqrt{3}}{(\eta_0-\epsilon)^3} \sum_{k_1 = 1}^{K} \sum_{k_2 = k_1}^{\min\{k_1 + 1, K\}} \left (\frac{\abs{G_{k_2}}}{n} \right)^{-\frac{1}{2}} \sqrt{\frac{n_{k_1}^4+(\frac{46}{3})n_{k_1}^3+8n_{k_1}^2}{(n_{k_1}-1)\cdots(n_{k_1}-4)}} \\
	&K_2(\eta) =\sum_{k = 1}^{K} \left(\frac{\abs{G_k}}{n}\right)^{\frac{1}{2}} \\
	&c = \max_{k \in \{1, \dots, K\}} \left (\frac{\abs{G_{k}}}{n} \right)^{-\frac{1}{2}}.
	\end{align*}
\end{corollary}

\section{A Multivariate Kolmogorov/Smooth Function Bound}\label{sec:mv.K.dist}
In this section we give a general upper bound on the multivariate \hbox{Kolmogorov} distance 
\begin{equation}\label{kol.d.def}
d_{kol}(W,Z) = \sup_{\bm{x}\in \mathbb{R}^p} \left|\mathbbm{1}\{W\le\bm{x}\}- \mathbbm{1}\{Z\le\bm{x}\}\right|
\end{equation}
of two $p$-dimensional random vectors $W,Z$, in terms of the smooth function class distances, such as those in our results in Sections~\ref{sec:main-result} and \ref{sec:exp-fam}.  In \eqref{kol.d.def}, $\mathbbm{1}\{W\le\bm{x}\}$ is the indicator function of the $p$-dimensional ``lower quadrant'' $\{W\le\bm{x}\}=\{W_1\le x_1,\ldots, W_p\le x_p\}$.  These bounds may be of independent interest, but another reason we include them here is because they open the door to applying our smooth function bounds for the group sequential MLEs to the normal, to Kolmogorov distance, which may be desirable in statistical applications.

First we modify our notation slightly to show more explicitly the dependence on dimension and derivative order. For
$\bm{k} = (k_1, \dots, k_p) \in \mathbb{N}_0^p$, let $\abs{\bm{k}} = \sum_{i = 1}^{p} k_i$,
and for functions $h:\R^p \rightarrow \R$ whose partial derivatives
\begin{equation*}
	h^{\bm{k}}(\bm{x}) = \frac{\partial^{k_1+\cdots+k_p}h}{\partial^{k_1}x_1\dots \partial^{k_p}x_p} \quad \text{exist for all} \quad 0 \leq \abs{\bm{k}} \leq m,
\end{equation*}
and $\norm{\cdot}$ the supremum norm, let $L_m^\infty(\R^p)$ be the collection of all functions $h:\R^p \rightarrow \R$ with
\begin{equation*}
	\norm{h}_{L_m^\infty(\R^p)} = \max_{0 \leq \abs{\bm{k}} \leq m} \norm{h^{(\bm{k})}}
\end{equation*}
finite. For random vectors $X$ and $Y$ in $\R^p$, letting
\begin{equation*}
	\mathcal{H}_{m, \infty, p} = \{h \in L_m^\infty(\R^p): \norm{h}_{L_m^\infty(\R^p)} \leq 1\},
\end{equation*}
define
\begin{equation*}
	d_{m, \infty, p}(X,Y) = \norm{\mathcal{L}(X)- \mathcal{L}(Y)}_{\mathcal{H}_{m, \infty, p}} = \sup_{h \in \mathcal{H}_{m, \infty, p}} \abs{\E{h(X)} - \E{h(Y)}}.
\end{equation*}
Connecting to the previous sections, note that, for example,  Theorem~\ref{my_result_1} gives a bound on $d_{3,\infty,q}$ between the standardized group sequential MLE and the $q$-dimensional standard normal, and Theorem~\ref{thm:my-result-1-improved} a bound on $d_{2,\infty,q}$.

We first state our general result in Theorem~\ref{thm:kol-smooth-fn-bound}, and then specialize to the $m=3$ case in Corollary~\ref{cor:kol.m3}.

\begin{theorem}\label{thm:kol-smooth-fn-bound}
	Suppose $W$ and $Z$ are $p$-dimensional, $p\ge 2$, random vectors such that $Z$ has a density bounded by a constant $C_1$. Then there exists a constant $C_2 \geq 1$ that depends only on $m$, $m \geq 1$, such that 
	\begin{equation*}
		d_{\text{kol}}(W, Z) \leq d_{m, \infty, p}(W, Z)^{\frac{p-1}{m+p-1}}\left(C_2^p +p + C_1 d_{m, \infty, p}(W, Z)^{\frac{1}{m+p-1}}\right).
	\end{equation*}
\end{theorem}
\begin{proof}
	Assume $d_{m, \infty, p}(W, Z) \leq 1$ since otherwise the bound is trivial. Consider $S_m: \R \rightarrow \R$ such that $S_m \in  L_m^\infty(\R)$ and
	\begin{equation*}
		S_m(x) = 
		\begin{cases}
		0 & \text{if }x \leq 0 \\
		f(x) & \text{if } 0 \leq x \leq 1 \\
		1 & \text{if } 1 \leq x
		\end{cases}
	\end{equation*}
where $f$ is some function with $0 \leq f(x) \leq 1$ for all $0 \leq x \leq 1$. Let $h_{z,\alpha}(x) = S_m(-x/\alpha + (\alpha + z)/\alpha)$ and notice that since $h_{z,\alpha}(x)$ dominates $\mathbbm{1}_{\{x \leq z\}}$,
	\begin{equation}\label{eq:first-inequality}
		\begin{aligned}
		P(W \leq \bm{z}) - P(Z \leq \bm{z} ) &\leq \E{\prod_{i=1}^{p}h_{z_i,\alpha}(W_i)} - \E{\prod_{i=1}^{p}h_{z_i,\alpha}(Z_i)} \\
		&+ \E{\prod_{i=1}^{p}h_{z_i,\alpha}(Z_i)} - P(Z \leq \bm{z}).
		\end{aligned}
	\end{equation}
	We begin by bounding the first two terms on the right hand side of \eqref{eq:first-inequality}. Since $h_{z_i,\alpha} \in L_m^\infty(\R)$, the function $\bm{h}_{\bm{z}, \alpha} := \prod_{i=1}^{p}h_{z_i,\alpha}$ is in $L_m^\infty(\R^p)$. We have
\begin{multline*}
\norm{\bm{h}_{\bm{z}, \alpha}}_{L_m^\infty(\R^p)} = \max_{0 \leq \abs{\bm{k}} \leq m} \norm{\bm{h}_{\bm{z}, \alpha}^{(\bm{k})}} = \max_{0 \leq \abs{\bm{k}} \leq m} \prod_i^p \norm{h_{z_i, \alpha}^{(k_i)}} \\
		=  \max_{0 \leq \abs{\bm{k}} \leq m} \prod_i^p \left ( \frac{1}{\alpha} \right )^{k_i} \norm{S_m^{(k_i)}} 
		 \leq \max_{0 \leq \abs{\bm{k}} \leq m} \prod_i^p \left ( \frac{1}{\alpha} \right )^{k_i} \norm{S_m}_{L_m^\infty(\R)} \\
		= \left (\frac{1}{\alpha} \right )^m\norm{S_m}_{L_m^\infty(\R)}^p.
\end{multline*}
	Thus
	\begin{equation*}
		\frac{\alpha^m \bm{h}_{\bm{z}, \alpha}}{\norm{S_m}_{L_m^\infty(\R)}^p} \in \mathcal{H}_{m, \infty, p}.
	\end{equation*}
Then the first two terms of the right hand side of \eqref{eq:first-inequality} are bounded by 
	\begin{equation}\label{eq:first-2-sums}
		\E{\prod_{i=1}^{p}h_{z_i,\alpha}(W_i)} - \E{\prod_{i=1}^{p}h_{z_i,\alpha}(Z_i)} \leq \frac{C_2^p d_{m, \infty, p}(W,Z)}{\alpha^m}
	\end{equation} 
	for $C_2 = \norm{S_m}_{L_m^\infty(\R)}$. Furthermore, $1 = \norm{S_m} \leq \norm{S_m}_{L_m^\infty(\R)} = C_2$.
	
	Now we consider the last two terms of \eqref{eq:first-inequality}. For $\bm{\alpha} = (\alpha, \dots, \alpha) \in \R^p$,
	\begin{equation}\label{eq:second-2-sums}
		\E{\prod_{i=1}^{p}h_{z_i,\alpha}(Z_i)} - P(Z \leq \bm{z}) \leq P(Z \leq \bm{z} + \bm{\alpha}) - P(Z \leq \bm{z}).
	\end{equation}
The right hand side of \eqref{eq:second-2-sums} can be bounded by a $(p+1)$-dimensional rectangle of base $\alpha^p$, height $C_1$, and $p$ `legs' parallel to each axis with base $\alpha^{p-1}$ and length $\Phi(z_i)$. Thus
	\begin{align}\label{eq:second-2-sums-final}
		P(Z \leq \bm{z} + \bm{\alpha}) - P(Z \leq \bm{z}) &\leq C_1 \alpha^p + \sum_{i = 1}^p \Phi(z_i)\alpha^{p-1} \nonumber\\
		&\leq C_1 \alpha^p + p\alpha^{p-1}.
	\end{align}
	Combining \eqref{eq:first-inequality} with \eqref{eq:first-2-sums} and \eqref{eq:second-2-sums-final} gives
	\begin{equation}
		P(W \leq \bm{z}) - P(Z \leq \bm{z} ) \leq \frac{C_2^p d_{m, \infty, p}(W,Z)}{\alpha^m} + C_1 \alpha^p + p\alpha^{p-1}.
	\end{equation}
	Letting $\alpha = d_{m, \infty, p}(W,Z)^{1/(m+p-1)}$ and plugging into the above equation yields,
	\begin{equation*}
		P(W \leq \bm{z}) - P(Z \leq \bm{z} ) \leq d_{m, \infty, p}(W,Z)^\frac{p-1}{m+p-1}\left(C_2^p + p + C_1 d_{m, \infty, p}(W,Z)^\frac{1}{m+p-1}\right).
	\end{equation*}
	Similarly,
	\begin{equation*}
		P(W \leq \bm{z}) - P(Z \leq \bm{z} ) \geq -d_{m, \infty, p}(W,Z)^\frac{p-1}{m+p-1}\left(C_2^p + p + C_1 d_{m, \infty, p}(W,Z)^\frac{1}{m+p-1}\right)
	\end{equation*}
	which completes the proof.
\end{proof}

\begin{corollary}\label{cor:kol.m3}
	The Kolmogorov distance is bounded by $d_{3, \infty, p}(\cdot, \cdot)$ as follows,
	\begin{equation*}
	d_{\text{kol}}(W, Z) \leq d_{3, \infty, p}(W, Z)^{\frac{p-1}{2+p}}\left(52.5^p +p + \frac{d_{3, \infty, p}(W, Z)^{\frac{1}{2+p}}}{(2\pi)^{p/2}}\right).
	\end{equation*}
\end{corollary}

\begin{proof}
We follow the proof of Theorem~\ref{thm:kol-smooth-fn-bound}. Let $S_m$ be the Hermite interpolation of some function $f$ between $f(0) = 0$ and $f(1)=1$ with $f^{(i)}(0)=f^{(i)}(1)=0$ for all $i=1, \dots, m$. Thus $S_m \in L_m^\infty(\R)$ with
	\begin{equation*}
	S_m(x) = 
	\begin{cases}
	0 & \text{if }x \leq 0 \\
	x^{m+1}\sum_{k=0}^{m}{m+k \choose k}{2m+1 \choose m-k}(-x)^k & \text{if } 0 \leq x \leq 1 \\
	1 & \text{if } 1 \leq x
	\end{cases}.
	\end{equation*}
	In particular,
	\begin{equation*}
		S_3(x) = -20x^7+70x^6-84x^5+35x^4 \quad \text{for} \quad 0 \leq x \leq 1.
	\end{equation*}
	Simple calculations yield $\norm{S_3}_{L_3^\infty(\R)} = 52.5$ and the result follows.
\end{proof}

\section{Conclusion and Future Directions}

We have generalized the results of \cite{AR_bound} and \cite{mvn} to find optimal order bounds for the joint distribution of a sequence of maximum likelihood estimates based on accumulating data.  The approximate normality of this joint distribution is an essential assumption underlying the statistics of group sequential hypothesis testing, the dominant paradigm in clinical trials. The specialization of this bound to exponentials families in Corollary~\ref{cor:exp-fam} is the simplest such bound covering this case in the non-sequential setting, that we are aware of.

A direction in which these results may be generalized is to examine log-likelihood functions of the form
\begin{equation*}
\sum_{i \in G_k}\log f_i(Y_i, \theta) + g_k(\mathcal{Y}_k, \theta),
\end{equation*}
with $\mathcal{Y}_k = \{Y_i : i \in G_k \}$. 
This form  may provide a way of  relaxing the independence assumption between samples that could be amenable to Stein's method. It is also the form of some log-likelihood functions of generalized linear mixed models (GLMMs) with random stage effects, for example,  the GLMM with Poisson response variable and canonical log link function. 


\section*{Appendix}
\appendix
\section{Proof of Theorem~\ref{my_result_1}}\label{sec:mle_non_asymptotic_proof}

In this appendix we provide the detailed proof of Theorem \ref{my_result_1}, followed by two auxiliary lemmas.
\begin{proof}[Proof of Theorem~\ref{my_result_1}]
	By the triangle inequality we will bound
	\begin{equation*} 
	\abs{\mathbb{E} [h(\sqrt{n} J^{-\frac{1}{2}}_n (\hat{\theta}^K - \theta_0^K))] - \mathbb{E} [ h(Z)]}
	\end{equation*}
	above by
	\begin{align} 
	&\abs {\mathbb{E} [h(\frac{1}{\sqrt{n}} \tilde{J}^{-\frac{1}{2}}_n [S(1, \theta_0), \dots, S(K, \theta_0)]^\intercal)] - \mathbb{E} [ h(Z)]} \label{main triangle inequality 1st term} \\
	+&\abs{\mathbb{E} [h(\sqrt{n} J^{-\frac{1}{2}}_n (\hat{\theta}^K - \theta_0))] - \mathbb{E} [h(\frac{1}{\sqrt{n}} \tilde{J}^{-\frac{1}{2}}_n [S(1, \theta_0), \dots, S(K, \theta_0)]^\intercal)]} \label{main triangle inequality 2nd term}.
	\end{align}
	where $S(k; \theta) = \sum_{i =1}^{n_k} S_i(Y_i ; \theta) \in \R^d$. We begin by finding an upper bound for \eqref{main triangle inequality 1st term}.
	\paragraph{Upper bound for \eqref{main triangle inequality 1st term}.}
	Let $\tilde{h}(w) = h(\tilde{J}^{-\frac{1}{2}}_n Aw)$ and define 
	\begin{gather*}
	S(G_k; \theta) = \sum_{i \in G_k} S_i(Y_i ; \theta) \in \R^d,\\
	W = n^{-\frac{1}{2}} [S(G_1, \theta_0)^\intercal \dots, S(G_K, \theta_0)^\intercal]^\intercal \in \R^q.
	\end{gather*}
	Then, $\tilde{h}(W) = h( n^{-\frac{1}{2}} \tilde{J}^{-\frac{1}{2}}_n [S(1, \theta_0)^\intercal \dots, S(K, \theta_0)^\intercal]^\intercal)$. Let
	\begin{align*}
	\Sigma = \Var{W} & = \frac{1}{n}\diag{\sum_{i \in G_1} I_i(\theta_0), \sum_{i \in G_2} I_i(\theta_0), \dots, \sum_{i \in G_K} I_i(\theta_0)} \\
	& = \diag{\bar{I}_n(G_1, \theta_0), \bar{I}_n(G_2, \theta_0), \dots, \bar{I}_n(G_K, \theta_0)}
	\end{align*}
	where we define $\bar{I}_n(G_k, \theta) = \frac{1}{n} \sum_{i \in G_k} I_i(\theta)$. Note a slight abuse of notation since $\bar{I}_n(k, \theta) = \frac{1}{n} \sum_{i=1}^{n_k} I_i(\theta)$. Thus, 
	\begin{equation*}
	\tilde{h}(\Sigma^\frac{1}{2}Z) = h(\tilde{J}^{-\frac{1}{2}}_n A \Sigma^\frac{1}{2}Z).
	\end{equation*}
	
	Now we claim $\tilde{J}^{-\frac{1}{2}}_n A \Sigma^\frac{1}{2} = 1_{q \times q}$. To see this note that, 
	\begin{align*}
	& \quad &A\Sigma A^\intercal &= \tilde{J}_n \\
	& \implies & A\Sigma^\frac{1}{2} &= \tilde{J}^{\frac{1}{2}}_n \\
	& \implies & \tilde{J}^{-\frac{1}{2}}_n A\Sigma^\frac{1}{2} &= 1_{q \times q}.
	\end{align*}
	Where the first equality is easily confirmed by matrix block multiplication. Thus,
	\[\tilde{h}(\Sigma^{\frac{1}{2}}Z) = h(Z).\]
	This implies that \eqref{main triangle inequality 1st term} is equal to
	\begin{equation*}
	\abs{\mathbb{E} [\tilde{h}(W)] - \mathbb{E} [\tilde{h}(\Sigma^{\frac{1}{2}}Z)]}.
	\end{equation*}
	By a similar argument as above we see that
	\begin{equation}\label{eq:Sigma_to_the_neg_1/2}
	\tilde{J}_n^{-1/2}A=\Sigma^{-1/2}= \diag{\bar{I}_n^{-1/2}(G_1;\theta_0), \dots, \bar{I}_n^{-1/2}(G_K;\theta_0)}
	\end{equation}
	
	Now we wish to apply the Stein's method bound found in Theorem \ref{RR_exchangeable_pair_bound} to the above equation. To do so we need $\mathbb{E}[W]$ and $\Var{W}$. From above we have that $\Var{W} = \Sigma$. To find $\mathbb{E}[W]$ note that,
	\begin{equation*}
	\mathbb{E} [W_{[i]}] = n^{-\frac{1}{2}} \mathbb{E} [S(G_i, \theta_0)] = n^{-\frac{1}{2}} \sum_{j \in G_i} \mathbb{E} [S_j (Y_j, \theta_0)] = 0_d.
	\end{equation*}
	Thus, $\mathbb{E} [W] = 0_q$.
	
	The next step in applying Theorem \ref{RR_exchangeable_pair_bound} is to find an appropriate exchangeable pair for $W$. Let $\{Y_i' : i = 1, \dots, n\}$ be an independent copy of $\{Y_i : i = 1, \dots, n\}$. Let $I \in \{1, \dots, n\}$ be a uniform r.v. independent of $\{Y_i'\}_{i = 1, \dots, n}$ and $\{Y_i\}_{i = 1, \dots, n}$. Then we replace $\frac{1}{\sqrt{n}}S_I(Y_I, \theta_0)$ by $\frac{1}{\sqrt{n}}S_I(Y_I', \theta_0)$ so
	\begin{equation*}
	W_{[\hat{k}]}' = W_{[\hat{k}]} - \frac{1}{\sqrt{n}}S_I(Y_I, \theta_0) + \frac{1}{\sqrt{n}}S_I(Y_I', \theta_0)
	\end{equation*}
	where $\hat{k}$ is the group containing observation $Y_I$ and $W_k' = W_k$ for $k \neq \hat{k}$. Then,
	\begin{align*}
	\mathbb{E} [W_{[k]}' - W_{[k]} | W] &= \sum_{j = 1}^{n} \mathbb{E} [W_{[k]}' - W_{[k]} | W, I = j] \mathbb{P}(I = j) \\
	&= \frac{1}{n} \sum_{j \in G_k} \mathbb{E} [W_{[k]}' - W_{[k]} | W, I = j] \\
	&= \frac{1}{n} \sum_{j \in G_k} \mathbb{E} [\frac{1}{\sqrt{n}} S_I(Y_I', \theta_0) - \frac{1}{\sqrt{n}} S_I(Y_I, \theta_0) | W, I = j] \\
	&= \frac{1}{n} \sum_{j \in G_k} \mathbb{E} [\frac{1}{\sqrt{n}} S_j(Y_j', \theta_0) - \frac{1}{\sqrt{n}} S_j(Y_j, \theta_0) | W] \\
	&= -\frac{1}{n} \sum_{j \in G_k} \mathbb{E} [\frac{1}{\sqrt{n}} S_j(Y_j, \theta_0) | W_{[k]}] \\
	&= -\frac{1}{n} \mathbb{E} [\sum_{j \in G_k} \frac{1}{\sqrt{n}} S_j(Y_j, \theta_0) | W_{[k]}] \\
	&= -\frac{1}{n} \mathbb{E} [W_{[k]} | W_{[k]}] \\
	&= -\frac{W_{[k]}}{n}
	\end{align*}
	Thus as in Theorem \ref{RR_exchangeable_pair_bound}, $\Lambda = \frac{1}{n}1_{q \times q}$, $R = 0_{q \times q}$, and it follows that
	\begin{align}
	&\abs{\mathbb{E} [\tilde{h}(W)] - \mathbb{E} [\tilde{h}(\Sigma^{\frac{1}{2}}Z)]}  \leq \nonumber\\
	& \frac{n \norm{\tilde{h}}_2}{4} \sum_{k_1, k_2 = 1}^{K}\sum_{i,j = 1}^{d} \sqrt{\Var{\mathbb{E} [([W_{[k_1]}']_i - [W_{[k_1]}]_i)([W_{[k_2]}']_j - [W_{[k_2]}]_j)| W]}} \label{RR1}\\
	&+\frac{n \norm{\tilde{h}}_3}{12} \sum_{k_1, k_2, k_3 = 1}^{K}\sum_{i,j,u = 1}^{d} \mathbb{E} \left|([W_{[k_1]}']_i - [W_{[k_1]}]_i)\right.\cdot \nonumber\\
		&\left.([W_{[k_2]}']_j - [W_{[k_2]}]_j)([W_{[k_3]}']_u - [W_{[k_3]}]_u)\right| \label{RR2}.
	\end{align}
	We now wish to bound and simplify both \eqref{RR1} and \eqref{RR2}. Notice that by construction for $k_1 \neq k_2$,
	\begin{equation*}
	([W_{[k_1]}']_i - [W_{[k_1]}]_i)([W_{[k_2]}']_j - [W_{[k_2]}]_j) = 0_d.
	\end{equation*}
	Similarly,
	\begin{equation*}
	([W_{[k_1]}']_i - [W_{[k_1]}]_i)([W_{[k_2]}']_j - [W_{[k_2]}]_j)([W_{[k_3]}']_u - [W_{[k_3]}]_u) = 0_d
	\end{equation*}
	unless $k_1 = k_2 = k_3$. Thus \eqref{RR1} is equal to
	\begin{equation}
	\frac{n \norm{\tilde{h}}_2}{4} \sum_{k = 1}^{K}\sum_{i = 1}^{d}\sum_{j = 1}^d \sqrt{\Var
		{\mathbb{E} [([W_{[k]}']_i - [W_{[k]}]_i)([W_{[k]}']_j - [W_{[k]}]_j)| W]}} \label{RR1.0.1}
	\end{equation}
	and \eqref{RR2} is equal to
	\begin{equation}
	\frac{n \norm{\tilde{h}}_3}{12} \sum_{k = 1}^{K}\sum_{i = 1}^{d}\sum_{j = 1}^d\sum_{u = 1}^{d} \mathbb{E} \abs{([W_{[k]}']_i - [W_{[k]}]_i)([W_{[k]}']_j - [W_{[k]}]_j)([W_{[k]}']_u - [W_{[k]}]_u)}. \label{RR2.0.1}
	\end{equation}
	
	In order to write $\norm{\tilde{h}}_2$ in terms of $\norm{h}_2$ and $\norm{\tilde{h}}_3$ in terms of $\norm{h}_3$ we use the chain rule. Let $x = \tilde{J}^{-\frac{1}{2}}_n Aw$ and define  $\abs{X} := \sup \limits_{i, j} \abs{X_{ij}}$ for any matrix $X$. Then,
	\begin{equation*}
	\frac{\partial \tilde{h}}{\partial w_j} = \sum_{i = 1}^{q} \frac{\partial h}{\partial x_i} \frac{\partial x_i}{\partial w_j} = \sum_{i = 1}^{q} \frac{\partial h}{\partial x_i} [\tilde{J}^{-\frac{1}{2}}_n A]_{ij}.
	\end{equation*}
	So,
	\begin{align*}
	\frac{\partial \tilde{h}^2}{\partial w_s \partial w_j} &= \frac{\partial}{\partial w_s } \sum_{i_1 = 1}^{q} \frac{\partial h}{\partial x_{i_1}} [\tilde{J}^{-\frac{1}{2}}_n A]_{i_1j} = \sum_{i_1 = 1}^{q} [\tilde{J}^{-\frac{1}{2}}_n A]_{i_1j} \sum_{i_2 = 1}^{q} \frac{\partial h^2}{\partial x_{i_2} \partial x_{i_1}} \frac{\partial x_{i_2}}{\partial w_s}\\
	&= \sum_{i_1 = 1}^{q} [\tilde{J}^{-\frac{1}{2}}_n A]_{i_1j} \sum_{i_2 = 1}^{q} \frac{\partial h^2}{\partial x_{i_2} \partial x_{i_1}} [\tilde{J}^{-\frac{1}{2}}_n A]_{i_2s} \\
	&\leq \sum_{i_1 = 1}^{q} \sum_{i_2 = 1}^{q} [\tilde{J}^{-\frac{1}{2}}_n A]_{i_1j}  [\tilde{J}^{-\frac{1}{2}}_n A]_{i_2s} \left | \frac{\partial h^2}{\partial x_{i_2} \partial x_{i_1}} \right | \\
	&\leq q^2 \abs{\tilde{J}^{-\frac{1}{2}}_n A}^2 \norm{h}_2 \\
	&= q^2 c^2 \norm{h}_2.
	\end{align*}
	Where the last equality holds by \eqref{eq:Sigma_to_the_neg_1/2} and we remind the reader that
	\begin{equation*}
	c = \max_{k \in \{1, \dots, K\}} \abs{\bar{I}^{-1/2}_n(G_k;\theta_0)}.
	\end{equation*}
	Similarly, 
	\begin{equation*}
	\frac{\partial \tilde{h}^3}{\partial w_u \partial w_s \partial w_j} \leq q^3 c^3 \norm{h}_3.
	\end{equation*}
	By taking the supremum of the left hand side of each of the above equations we get,
	\begin{gather*}
	\norm{\tilde{h}}_2 \leq q^2 c^2 \norm{h}_2, \\
	\norm{\tilde{h}}_3 \leq q^3 c^3 \norm{h}_3.
	\end{gather*}
	Thus combining the above with \eqref{RR1.0.1} and \eqref{RR2.0.1} we see that \eqref{RR1} is less than or equal to
	\begin{equation}
	\frac{n q^2 c^2 \norm{h}_2}{4} \sum_{k = 1}^{K} \sum_{i = 1}^{d}\sum_{j = 1}^d \bigg \{ \\ \Var{\mathbb{E} [([W_{[k]}']_i - [W_{[k]}]_i)([W_{[k]}']_j - [W_{[k]}]_j)| W]} \bigg \}^\frac{1}{2} \label{RR1.1}
	\end{equation}
	and \eqref{RR2} is less than or equal to
	\begin{equation}
	\frac{n q^3 c^3 \norm{h}_3}{12} \sum_{k = 1}^{K} \sum_{i = 1}^{d}\sum_{j = 1}^d\sum_{u = 1}^{d} \mathbb{E} \abs{([W_{[k]}']_i - [W_{[k]}]_i) \cdot \\ ([W_{[k]}']_j - [W_{[k]}]_j)([W_{[k]}']_u - [W_{[k]}]_u)}. \label{RR2.1}
	\end{equation}
	
	To bound the variance of the conditional expectations in \eqref{RR1.1}, we remind the reader of the notation
	\begin{gather*}
	\xi_{ij} = [n^{-\frac{1}{2}} S_i(Y_i, \theta_0)]_j \\	
	\xi_{ij}' = [n^{-\frac{1}{2}} S_i(Y_i', \theta_0)]_j.
	\end{gather*}
	We also denote
	\begin{gather*}
	C_2 = \frac{q^2 c^2 \norm{h}_2}{4} \\
	C_3 = \frac{q^3 c^3 \norm{h}_3}{12}.
	\end{gather*}
	
	Now let $\mathcal{A} = \sigma(Y_1, Y_2, \dots, Y_n)$. Then since $\sigma(W) \subset \mathcal{A}$, for any r.v. $X$, $\Var{ \mathbb{E} [X | W]} \leq \Var{ \mathbb{E} [X | \mathcal{A}]}$. Thus \eqref{RR1.1} is less than or equal to
	\begin{multline}
n C_2 \sum_{k = 1}^{K} \Big\{ \sum_{j = 1}^d \sqrt{\Var{\mathbb{E} [(\xi_{Ij}' - \xi_{Ij})^2 \mathbbm{1}\{I \in G_k \}| \mathcal{A}]}}\\
	+ 2 \sum_{i < j} \sqrt{\Var{\mathbb{E} [(\xi_{Ii}' - \xi_{Ii})(\xi_{Ij}' - \xi_{Ij}) \mathbbm{1}\{I \in G_k \}| \mathcal{A}]}}\Big\}. \label{square terms and cross terms}
\end{multline}
	
	Since $\{Y_i'\}_{i = 1, \dots, n}$ is an independent copy of $\{Y_i\}_{i = 1, \dots, n}$, $\xi_{ik}'$ is independent of $\mathcal{A}$, and $I$ is independent of $\{\xi_{Ik}', \xi_{Ik}\}$, \eqref{square terms and cross terms} becomes
	\begin{multline} \label{RR1.2}
n C_2 \sum_{k = 1}^{K} \sum_{j = 1}^d \left\{ \mathrm{Var} \left[\mathbb{E} [(\xi_{Ij}')^2 \mathbbm{1} \{I \in G_k \}] -2 n^{-1} \sum_{i \in G_k} \mathbb{E} [\xi_{ij}'] \mathbb{E} [\xi_{ij}| \mathcal{A}] \right.\right.\\
\left.\left.+ n^{-1} \sum_{i \in G_k} \mathbb{E} [\xi_{ij}^2| \mathcal{A}] \right] \right\}^{\frac{1}{2}}\\
	+ 2 \sum_{i < j} \left[\mathrm{Var} \left[\mathbb{E} [\xi_{Ii}'\xi_{Ij}' \mathbbm{1} \{I \in G_k \}]  - n^{-1} \sum_{v \in G_k} \mathbb{E} [\xi_{vi}'] \mathbb{E} [\xi_{vj}| \mathcal{A}] \right.\right.\\
	\left.\left. - n^{-1} \sum_{v \in G_k} \mathbb{E} [\xi_{vj}'] \mathbb{E} [\xi_{vi}| \mathcal{A}] + n^{-1} \sum_{v \in G_k} \mathbb{E} [\xi_{vi} \xi_{vj}| \mathcal{A}] \right] \right]^{\frac{1}{2}}.
\end{multline}

	Using that $\mathbb{E} [\xi_{ij}'] = 0$ and independence we get \eqref{RR1.2} is equal to,
	\begin{multline}
	n C_2 \sum_{k = 1}^{K} \Big\{ \sum_{j = 1}^d \Big [ \frac{1}{n^2} \Var{\sum_{i \in G_k} \mathbb{E} [\xi_{ij}^2 | \mathcal{A}]} \Big ]^{\frac{1}{2}} \\
	+ 2 \sum_{i < j} \Big [ \frac{1}{n^2} \mathrm{Var} [ \sum_{v \in G_k} \mathbb{E} [\xi_{vi} \xi_{vj}| \mathcal{A}] ] \Big ]^ {\frac{1}{2}} \Big \} \\
	= C_2 \sum_{k = 1}^{K} \Big\{ \sum_{j = 1}^d \Big [ \Var{\sum_{i \in G_k} \xi_{ij}^2} \Big ]^{\frac{1}{2}} + 2 \sum_{i < j} \Big [ \mathrm{Var} [ \sum_{v \in G_k} \xi_{vi} \xi_{vj}] \Big ]^ {\frac{1}{2}} \Big \} \\
	= C_2 \sum_{k = 1}^{K} \Big\{ \sum_{j = 1}^d \Big [ \sum_{i \in G_k} \Var{\xi_{ij}^2} \Big ]^{\frac{1}{2}} + 2 \sum_{i < j} \Big [ \sum_{v \in G_k} \mathrm{Var} [ \xi_{vi} \xi_{vj}] \Big ]^ {\frac{1}{2}} \Big \} \\
	= \frac{q^2 c^2 \norm{h}_2}{4} K_2(\theta_0)
	\end{multline}
	where $K_2(\theta_0)$ is defined in the theorem statement. Next we simplify \eqref{RR2.1} by noticing that it is equal to
	\begin{align*}
	&n C_3 \sum_{k = 1}^{K}\sum_{i = 1}^{d}\sum_{j = 1}^d\sum_{u = 1}^{d} \mathbb{E} \abs{(\xi'_{Ii} - \xi_{Ii})(\xi'_{Ij} - \xi_{Ij})(\xi'_{Iu} - \xi_{Iu})} \nonumber \\
	&=n C_3 \sum_{k = 1}^{K}\sum_{i = 1}^{d}\sum_{j = 1}^d\sum_{u = 1}^{d} \frac{1}{n} \sum_{v \in G_k} \mathbb{E} \abs{(\xi'_{vi} - \xi_{vi})(\xi'_{vj} - \xi_{vj})(\xi'_{vu} - \xi_{vu})} \nonumber \\
	&= C_3 \sum_{k = 1}^{K} \sum_{v \in G_k} \mathbb{E} \Big [ \sum_{i = 1}^{d}\sum_{j = 1}^d\sum_{u = 1}^{d} \abs{(\xi'_{vi} - \xi_{vi})} \abs{(\xi'_{vj} - \xi_{vj})} \abs{(\xi'_{vu} - \xi_{vu})} \Big ] \nonumber \\
	&= C_3 \sum_{k = 1}^{K} \sum_{v \in G_k} \mathbb{E} \Big [ \sum_{j = 1}^d \abs{(\xi'_{vj} - \xi_{vj})}  \Big ]^3 \nonumber \\
	&= C_3 \sum_{i = 1}^{n} \mathbb{E} \Big [ \sum_{j = 1}^d \abs{(\xi'_{ij} - \xi_{ij})}  \Big ]^3 \nonumber \\
	&= \frac{ q^3 c^3 \norm{h}_3}{12} K_3(\theta_0)
	\end{align*}
	where again the definition of $K_3(\theta_0)$ can be found above in the statement of the theorem. Therefore we have upper bounded the first term in our main triangle inequality~\eqref{main triangle inequality 1st term} by
	\begin{equation}\label{main triangle inequality 1st term bound}
	\frac{q^2 c^2 \norm{h}_2}{4} K_2(\theta_0) + \frac{ q^3 c^3 \norm{h}_3}{12} K_3(\theta_0).
	\end{equation}
	Next we will upper bound the remaining term in our main triangle inequality~\eqref{main triangle inequality 2nd term}.
	
	\paragraph{Upper bound for \eqref{main triangle inequality 2nd term}.} Define
	\begin{equation*}
	S := [S(1, \theta_0)^\intercal, S(2, \theta_0)^\intercal, \dots, S(K, \theta_0)^\intercal]^\intercal.
	\end{equation*} 
	A second-order Taylor expansion of \\ $[S(1, \theta)^\intercal, S(2, \theta)^\intercal, \dots, S(K, \theta)^\intercal]_j^\intercal$ about $\theta_0^K$ evaluated at $\hat{\theta}^K$ yields,
	\begin{multline}
	[S(1, \hat{\theta}_1), S(2, \hat{\theta}_2), \dots, S(K, \hat{\theta}_K)]_j^\intercal = S_j + \sum_{i = 1}^{q} [\hat{\theta}^K - \theta_0^K]_i \Big (\frac{\partial}{\partial \theta_i} S_j\Big) + \\
	\frac{1}{2} \sum_{i = 1}^{q}\sum_{u = 1}^{q} [\hat{\theta}^K - \theta_0^K]_i [\hat{\theta}^K - \theta_0^K]_u \Big (\frac{\partial^2}{\partial \theta_i\partial \theta_u} [S(1, \theta), S(2, \theta), \dots, S(K, \theta)]^\intercal_j \Big \rvert_{\theta = \theta_0^*}\Big).
	\end{multline}
	We clarify a slight abuse of notation for the reader that in order to simplify the expressions in this section we use $\frac{\partial}{\partial \theta_i}$ to signify $\frac{\partial}{\partial (\theta^K)_i}$ where $\theta^K = [\theta_1, \dots, \theta_q]^\intercal$ and for fixed $k$, $S(k, \theta)$ is a function of $\theta_{d(k-1) + 1}, \dots, \theta_{dk}$. So for example $\frac{\partial}{\partial \theta_i} S_j = 0$ if  $j \in \{d(k_1-1) +1, \dots, dk_1\}$ and $i \in \{d(k_2-1) +1, \dots, dk_2\}$ for $k_1 \neq k_2$. 
	
	Define $J_n^* = J_n^*(\theta) = \mathrm{diag}(\bar{I}_n(1, \theta), \bar{I}_n(2, \theta), \dots, \bar{I}_n(K, \theta)) \in \R^{q \times q}$. Noticing that \\ $[S(1, \hat{\theta}_1), S(2, \hat{\theta}_2), \dots, S(K, \hat{\theta}_K)]_j^\intercal = 0$ and adding $\sum_{i = 1}^{q} n[J_n^*]_{ji} [\hat{\theta}^K - \theta_0^K]_i$ to both sides yields, 
	\begin{multline*}
	\sum_{i = 1}^{q} n[J_n^*]_{ji} [\hat{\theta}^K - \theta_0^K]_i = S_j + \sum_{i = 1}^{q} [\hat{\theta}^K - \theta_0^K]_i \Big (\frac{\partial}{\partial \theta_i} S_j + n[J_n^*]_{ji}\Big) + \\
	\frac{1}{2} \sum_{i = 1}^{q}\sum_{u = 1}^{q} [\hat{\theta}^K - \theta_0^K]_i [\hat{\theta}^K - \theta_0^K]_u \Big (\frac{\partial^2}{\partial \theta_i\partial \theta_u} [S(1, \theta), S(2, \theta), \dots, S(K, \theta)]^\intercal_j \Big \rvert_{\theta = \theta_0^*}\Big).
	\end{multline*}
	Using the above which holds $\forall j \in \{1, 2, \dots, q\}$ gives,
	\begin{multline}\label{n_J_star}
	n J_n^* (\hat{\theta}^K - \theta_0^K) = S + \sum_{i = 1}^{q} [\hat{\theta}^K - \theta_0^K]_i \Big (\frac{\partial}{\partial \theta_i} S + n[J_n^*]_{*i}\Big) + \\
	\frac{1}{2} \sum_{i = 1}^{q}\sum_{u = 1}^{q} [\hat{\theta}^K - \theta_0^K]_i [\hat{\theta}^K - \theta_0^K]_u \Big (\frac{\partial^2}{\partial \theta_i\partial \theta_u} [S(1, \theta), S(2, \theta), \dots, S(K, \theta)]^\intercal \Big \rvert_{\theta = \theta_0^*}\Big).
	\end{multline}
	
	Recalling the definitions of $\tilde{J}_n$ and $J_n$ we will show that  $\tilde{J}_n^{-\frac{1}{2}} = J_n^{-\frac{1}{2}} {J_n^*}^{-1}$. To demonstrate this equality we first observe that $\tilde{J}_n = J_n^* J_n J_n^*$ follows from simple block matrix multiplication. Then,
	\begin{align*}
	\tilde{J}_n &= J_n^* J_n J_n^* \\
	\implies \tilde{J}_n^{-1} &= {J_n^*}^{-1} J_n^{-1} {J_n^*}^{-1} \\
	\implies J_n^{-1} &= {J_n^*} \tilde{J}_n^{-1} {J_n^*} \\
	\implies J_n^{-\frac{1}{2}} &= \tilde{J}_n^{-\frac{1}{2}} {J_n^*} \\
	\implies \tilde{J}_n^{-\frac{1}{2}} &= J_n^{-\frac{1}{2}} {J_n^*}^{-1}.
	\end{align*}
	Thus by multiplying both sides of \eqref{n_J_star} by $\tilde{J}_n^{-\frac{1}{2}}/\sqrt{n}$ we get
	\begin{multline} \label{taylor form of (hat theta^K - theta^K)}
	\sqrt{n} J_n^{-\frac{1}{2}} (\hat{\theta}^K - \theta_0^K) = \frac{\tilde{J}_n^{-\frac{1}{2}}}{\sqrt{n}} \Big \{ S + \sum_{i = 1}^{q} [\hat{\theta}^K - \theta_0^K]_i \Big (\frac{\partial}{\partial \theta_i} S + n[J_n^*]_{*i}\Big) + \\
	\frac{1}{2} \sum_{i = 1}^{q}\sum_{u = 1}^{q} [\hat{\theta}^K - \theta_0^K]_i [\hat{\theta}^K - \theta_0^K]_u \Big (\frac{\partial^2}{\partial \theta_i\partial \theta_u} [S(1, \theta), S(2, \theta), \dots, S(K, \theta)]^\intercal \Big \rvert_{\theta = \theta_0^*}\Big) \Big \}.
	\end{multline}
	
	Recall that $Q_i := [\hat{\theta}^K - \theta_0^K]_i$ and define
	\begin{equation*}
	R := \frac{\tilde{J}_n^{-\frac{1}{2}}}{2 \sqrt{n}} \sum_{i = 1}^{q}\sum_{u = 1}^{q} Q_i Q_u \Big (\frac{\partial^2}{\partial \theta_i\partial \theta_u} [S(1, \theta), S(2, \theta), \dots, S(K, \theta)]^\intercal \Big \rvert_{\theta = \theta_0^*}\Big).
	\end{equation*}
	Notice that $R$ is the last term of the right hand side of \eqref{taylor form of (hat theta^K - theta^K)}. Furthermore $R$ has no relation to the remainder matrix that appears in Theorem \ref{RR_exchangeable_pair_bound} which will not be used within the current proof. Let
	\begin{align*}
	T_1 &:= h \big (\sqrt{n} J_n^{- \frac{1}{2}} (\hat{\theta}^K - \theta_0^K) \big ) - h(\frac{1}{\sqrt{n}} \tilde{J}_n^{- \frac{1}{2}} S + R) \\
	T_2 &:= h(\frac{1}{\sqrt{n}} \tilde{J}_n^{- \frac{1}{2}} S + R) - h(\frac{1}{\sqrt{n}} \tilde{J}_n^{- \frac{1}{2}} S).
	\end{align*}
	Then \eqref{main triangle inequality 2nd term} is equal to
	\[\abs{\mathbb{E} T_1 + \mathbb{E} T_2} \leq \mathbb{E} \abs{T_1} + \mathbb{E} \abs{T_2}.\]
	
	We first address an upper bound for $T_1$. Using a first degree multivariate Taylor series approximation of $h$ evaluated at $\sqrt{n} J_n^{- \frac{1}{2}} (\hat{\theta}^K - \theta_0^K)$ and centered at $n^{-\frac{1}{2}} \tilde{J}_n^{- \frac{1}{2}} S + R$ we get,
	\begin{multline*}
	h \big (\sqrt{n} J_n^{- \frac{1}{2}} (\hat{\theta}^K - \theta_0^K) \big ) = h(n^{-\frac{1}{2}} \tilde{J}_n^{- \frac{1}{2}} S + R) + \\
	\sum_{j = 1}^{q} \frac{\partial}{\partial \theta_j}h(x^*) \Big ( \sqrt{n} J_n^{- \frac{1}{2}} (\hat{\theta}^K - \theta_0^K) - n^{-\frac{1}{2}} \tilde{J}_n^{- \frac{1}{2}} S - R \Big )_{j*},
	\end{multline*}
	where in this case $j*$ indicates the $j^{\text{th}}$ row. Simplifying yields
	\begin{equation*}
	\abs{T_1} \leq \norm{h}_1 \sum_{j = 1}^{q} \abs{ \sqrt{n} [J_n^{- \frac{1}{2}}]_{j*} (\hat{\theta}^K - \theta_0^K) - n^{-\frac{1}{2}} [\tilde{J}_n^{- \frac{1}{2}}]_{j*} S - R_{j*} }.
	\end{equation*}
	
	Next we will write $\sqrt{n} [J_n^{- \frac{1}{2}}]_{j*} (\hat{\theta}^K - \theta_0^K)$ in terms of $n^{-\frac{1}{2}} [\tilde{J}_n^{- \frac{1}{2}}]_{j*} S$ and two remainder terms. Using $\eqref{taylor form of (hat theta^K - theta^K)}$ component wise and substituting into the above for $\sqrt{n} [J_n^{- \frac{1}{2}}]_{j*} (\hat{\theta}^K - \theta_0^K)$ gives us
	\begin{align*}
	\abs{T_1} &\leq \norm{h}_1 \sum_{j = 1}^{q} \Big \lvert n^{-\frac{1}{2}} [\tilde{J}_n^{- \frac{1}{2}}]_{j*} \sum_{i = 1}^{q} [\hat{\theta}^K - \theta_0^K]_i \Big (\frac{\partial}{\partial \theta_i} S + n[J_n^*]_{*i}\Big) \Big \rvert \\
	&\leq \frac{\norm{h}_1}{\sqrt{n}} \sum_{j = 1}^{q} \sum_{l = 1}^{q} \abs{[\tilde{J}_n^{- \frac{1}{2}}]_{jl}} \sum_{i = 1}^{q} \Big \lvert [\hat{\theta}^K - \theta_0^K]_i \Big ([\frac{\partial}{\partial \theta_i} S]_l + n[J_n^*]_{li}\Big) \Big \rvert.
	\end{align*}
	Taking the expected value and applying the Cauchy-Schwartz inequality we get
	\begin{equation*}
	\mathbb{E} \abs{T_1} \leq \frac{\norm{h}_1}{\sqrt{n}} \sum_{j = 1}^{q} \sum_{l = 1}^{q} \abs{[\tilde{J}_n^{- \frac{1}{2}}]_{jl}} \sum_{i = 1}^{q} \sqrt{ \mathbb{E} [[\hat{\theta}^K - \theta_0^K]_i]^2 \mathbb{E} \Big ([\frac{\partial}{\partial \theta_i} S]_l + n[J_n^*]_{li}\Big)^2 }.
	\end{equation*} 
	By writing the matrices in block form and removing the zero summands, the above becomes
	\begin{multline} \label{T1 Bound}
	\mathbb{E} \abs{T_1} \leq \frac{\norm{h}_1}{\sqrt{n}} \sum_{k_1 = 1}^{K} \sum_{k_2 = 1}^{K} \sum_{j = 1}^{d} \sum_{l = 1}^{d} \abs{[[\tilde{J}_n^{- \frac{1}{2}}]_{[k_2][k_1]}]_{jl}} \Big \{ \\
	\sum_{i = 1}^{d} \sqrt{ \mathbb{E} [[\hat{\theta}_{k_1} - \theta_0]_i]^2 \mathbb{E} \Big (\frac{\partial^2}{\partial \theta_i \partial \theta_l} \ell_{k_1}(\theta_0, Y) + n\bar{I}_n(k_1, \theta_0)_{li}\Big)^2 }   \Big \}.
	\end{multline} 
	
	We now address an upper bound for $T_2$. By the law of total probability and using $Q_{(m)} =\max_{i \in \{1, \dots, q\}} Q_i$ we see that
	\begin{align} \label{E T2}
	\mathbb{E} \abs{T_2} &= \mathbb{E} [\abs{T_2} \mid \abs{Q_{(m)}} \geq \epsilon] \mathbb{P} (\abs{Q_{(m)}} \geq \epsilon) + \mathbb{E} [\abs{T_2} \mid \abs{Q_{(m)}} < \epsilon] \mathbb{P} (\abs{Q_{(m)}} < \epsilon) \nonumber\\
	&\leq 2 \norm{h} \mathbb{P} (\abs{Q_{(m)}} \geq \epsilon) + \mathbb{E} [\abs{T_2} \mid \abs{Q_{(m)}} < \epsilon] \nonumber \\
	&\leq \frac{2\norm{h}}{\epsilon^2} \mathbb{E} Q_{(m)}^2 + \mathbb{E} [\abs{T_2} \mid \abs{Q_{(m)}} < \epsilon] \nonumber \\
	&\leq \frac{2\norm{h}}{\epsilon^2} \mathbb{E} \big [\sum_{j = 1}^{q}Q_{j}^2 \big] + \mathbb{E} [\abs{T_2} \mid \abs{Q_{(m)}} < \epsilon],
	\end{align}
	where the third line follows from Markov's inequality. To bound $\mathbb{E} [\abs{T_2} \mid \abs{Q_{(m)}} < \epsilon]$ we again use a first order Taylor expansion of $h$ evaluated at $(n \tilde{J}_n)^{- \frac{1}{2}} S + R$ and centered at $(n\tilde{J}_n)^{- \frac{1}{2}} S$ which gives,
	\begin{multline} \label{T2}
	\abs{T_2} \leq \norm{h}_1 \sum_{j = 1}^{q} \abs{R_{j*}} \\
	= \frac{\norm{h}_1}{2\sqrt{n}} \sum_{l = 1}^{q} \sum_{j = 1}^{q} \abs{[\tilde{J}_n^{-\frac{1}{2}}]_{lj}}\sum_{i = 1}^{q}\sum_{u = 1}^{q} Q_i Q_u \Big (\frac{\partial^2}{\partial \theta_i\partial \theta_u} [S(1, \theta), \dots, S(K, \theta)]^\intercal_j \Big \rvert_{\theta = \theta_0^*}\Big).
	\end{multline}
	Combining \eqref{E T2} and \eqref{T2} we get
	\begin{multline*}
	\mathbb{E} \abs{T_2} = \frac{2\norm{h}}{\epsilon^2} \mathbb{E} \big [\sum_{j = 1}^{q}Q_{j}^2 \big] + 
	\frac{\norm{h}_1}{2\sqrt{n}} \sum_{l = 1}^{q} \sum_{j = 1}^{q} \abs{[\tilde{J}_n^{-\frac{1}{2}}]_{lj}}\sum_{i = 1}^{q}\sum_{u = 1}^{q} \Big \{ \\
	\mathbb{E} \Big [ Q_i Q_u \Big (\frac{\partial^2}{\partial \theta_i\partial \theta_u} [S(1, \theta), \dots, S(K, \theta)]^\intercal_j \Big \rvert_{\theta = \theta_0^*}\Big) \mid \abs{Q_{(m)}} < \epsilon \Big ] \Big\}.
	\end{multline*}
	Using the Cauchy-Schwartz inequality and \citet[][Lemma~4.1]{mvn}, 
	\begin{multline*}
	\mathbb{E} \abs{T_2} \leq \frac{2\norm{h}}{\epsilon^2} \mathbb{E} \big [\sum_{j = 1}^{q}Q_{j}^2 \big] + 
	\frac{\norm{h}_1}{2\sqrt{n}} \sum_{l = 1}^{q} \sum_{j = 1}^{q} \abs{[\tilde{J}_n^{-\frac{1}{2}}]_{lj}}\sum_{i = 1}^{q}\sum_{u = 1}^{q} \Big \{ \\
	(\mathbb{E} [Q_i^2 Q_u^2])^{\frac{1}{2}} (\mathbb{E} [\big (\frac{\partial^2}{\partial \theta_i\partial \theta_u} [S(1, \theta), \dots, S(K, \theta)]^\intercal_j \Big \rvert_{\theta = \theta^{\dagger}} \big )^2 \mid \abs{Q_{(m)}} < \epsilon ] )^\frac{1}{2}\Big\}.
	\end{multline*}
	Now by writing the matrices in block form, the above becomes
	\begin{multline*}
	\mathbb{E} \abs{T_2} \leq \frac{2\norm{h}}{\epsilon^2} \mathbb{E} \big [\sum_{j = 1}^{q}Q_{j}^2 \big] + 
	\frac{\norm{h}_1}{2\sqrt{n}} \sum_{k_1 = 1}^{K} \sum_{k_2 = 1}^{K} \sum_{l = 1}^{d} \sum_{j = 1}^{d} \abs{[[\tilde{J}_n^{-\frac{1}{2}}]_{[k_2][k_1]}]_{lj}} \Big \{ \\ 
	\sum_{i = 1}^{q}\sum_{u = 1}^{q} (\mathbb{E} [Q_i^2 Q_u^2])^{\frac{1}{2}} (\mathbb{E} [\big (\frac{\partial^2}{\partial \theta_i\partial \theta_u} S(k_1;\theta)_{j} \Big \rvert_{\theta = \theta^{\dagger}} \big )^2 \mid \abs{Q_{(m)}} < \epsilon ] )^\frac{1}{2}\Big\}.
	\end{multline*}
	We notice that $\frac{\partial^2}{\partial \theta_i\partial \theta_u} S(k_1, \theta)_j \Big \rvert_{\theta = \theta^{\dagger}}$ is equal to $0$ unless $\theta_i$ and $\theta_u$ are both corresponding to $S(k_1, \theta)$. Thus we simplify to get, 
	\begin{multline} \label{T2-bound-J-tilde}
	\mathbb{E} \abs{T_2} \leq \frac{2\norm{h}}{\epsilon^2} \mathbb{E} \big [\sum_{j = 1}^{q}Q_{j}^2 \big] + 
	\frac{\norm{h}_1}{2\sqrt{n}} \sum_{k_1 = 1}^{K} \sum_{k_2 = 1}^{K} \sum_{l = 1}^{d} \sum_{j = 1}^{d} \abs{[[\tilde{J}_n^{-\frac{1}{2}}]_{[k_2][k_1]}]_{lj}} \Big \{ \\ 
	\sum_{i = 1}^{d}\sum_{u = 1}^{d} \big (\mathbb{E} \big \{([\hat{\theta}_{k_1} - \theta_0]_i)^2 ([\hat{\theta}_{k_1} - \theta_0]_u)^2 \big \} \big )^{\frac{1}{2}} \cdot \\ (\mathbb{E} [ \big (\frac{\partial^3}{\partial \theta_i \partial \theta_u \partial \theta_j} \ell_{k_1}(\theta, Y) \Big \rvert_{\theta = \theta_0^*} \big)^2 \mid \abs{Q_{(m)}} < \epsilon ] )^\frac{1}{2}\Big\} \\
	\leq \frac{2\norm{h}}{\epsilon^2} \mathbb{E} \big [\sum_{j = 1}^{q}Q_{j}^2 \big] + 
	\frac{\norm{h}_1}{2\sqrt{n}} \sum_{k_1 = 1}^{K} \sum_{k_2 = 1}^{K} \sum_{l = 1}^{d} \sum_{j = 1}^{d} \abs{[[\tilde{J}_n^{-\frac{1}{2}}]_{[k_2][k_1]}]_{lj}} \Big \{ \\ 
	\sum_{i = 1}^{d}\sum_{u = 1}^{d} \big (\mathbb{E} \big \{([\hat{\theta}_{k_1} - \theta_0]_i)^2 ([\hat{\theta}_{k_1} - \theta_0]_u)^2 \big \} \big )^{\frac{1}{2}} (\mathbb{E} [ \big (M^{k_1}_{iuj}(Y) \big)^2 \mid \abs{Q_{(m)}} < \epsilon ] )^\frac{1}{2}\Big\}.
	\end{multline}
	
	As shown above $\tilde{J}_n^{-\frac{1}{2}}A\Sigma^{\frac{1}{2}} = 1_{q\times q}$ and so $\tilde{J}_n^{-\frac{1}{2}} = \Sigma^{-\frac{1}{2}}A^{-1}$. Inverting $A$ and simple block matrix multiplication yields
	\begin{equation}
		[\tilde{J}_n^{-\frac{1}{2}}]_{[k_2][k_1]} = 
		\begin{cases}
		\bar{I}_n^{-\frac{1}{2}}(G_{k_2}; \theta_0) & k_2 = k_1 \\
		-\bar{I}_n^{-\frac{1}{2}}(G_{k_2}; \theta_0) & k_2 = k_1+1 \\
		0_{d \times d} & \text{ otherwise }.
		\end{cases}
	\end{equation}
	Hence, \eqref{T2-bound-J-tilde} simplifies to
	\begin{multline} \label{T2-bound}
	\mathbb{E} \abs{T_2} \leq \frac{2\norm{h}}{\epsilon^2} \mathbb{E} \big [\sum_{j = 1}^{q}Q_{j}^2 \big] + 
	\frac{\norm{h}_1}{2\sqrt{n}} \sum_{k_1 = 1}^{K} \sum_{k_2 = k_1}^{\min\{k_1 + 1, K\}} \sum_{l = 1}^{d} \sum_{j = 1}^{d} \abs{\bar{I}_n^{-\frac{1}{2}}(G_{k_2}; \theta_0)_{lj}} \Big \{ \\ 
	\sum_{i = 1}^{d}\sum_{u = 1}^{d} \big (\mathbb{E} \big \{([\hat{\theta}_{k_1} - \theta_0]_i)^2 ([\hat{\theta}_{k_1} - \theta_0]_u)^2 \big \} \big )^{\frac{1}{2}} (\mathbb{E} [ \big (M^{k_1}_{iuj}(Y) \big)^2 \mid \abs{Q_{(m)}} < \epsilon ] )^\frac{1}{2}\Big\}.
	\end{multline}
	
	Therefore from an analogous simplification of $\tilde{J}_n^{-\frac{1}{2}}$ made to \eqref{T1 Bound} along with \eqref{T2-bound} we have \eqref{main triangle inequality 2nd term} is bounded above by
	\begin{equation} \label{main triangle inequality 2nd term bound}
	\frac{2\norm{h}}{\epsilon^2} \mathbb{E} \big [\sum_{j = 1}^{q}Q_{j}^2 \big] + \frac{\norm{h}_1}{\sqrt{n}} K_1(\theta_0).
	\end{equation}
	Combining \eqref{main triangle inequality 2nd term bound} and \eqref{main triangle inequality 1st term bound} gives the desired bounds in \eqref{my_result_1_statement}.
\end{proof}

In the following lemma we use multi-index notation. A multi-index in an $n$-tuple of nonnegative integers 

\[\alpha = (\alpha_1, \alpha_2, \dots, \alpha_n) \qquad (\alpha_j \in \{0, 1, 2, \dots\}).\]
If $\alpha$ is a multi-index then,
\begin{gather*}
\abs{\alpha} = \alpha_1 + \alpha_2 + \dots \alpha_n, \qquad \alpha! = \alpha_1!\alpha_2!\dots\alpha_n!\, \\ 
\bm{x}^\alpha = x_1^{\alpha_1}x_2^{\alpha_2} \dots x_n^{\alpha_n} \quad (\bm{x} \in \R^n), \\
\partial^\alpha f = \frac{\partial^{\abs{\alpha}}f}{\partial x_1^{\alpha_1} \partial x_2^{\alpha_2} \dots \partial x_n^{\alpha_n}}.
\end{gather*}

\section{Proof of Corollary~\ref{cor:exp-fam}}
To prove the corollary we solve for $K_1(\eta)$, $K_2(\eta)$, and $K_3(\eta)$ from Theorem~\ref{my_result_1}. Since the observations are i.i.d., the information matrices simplify to
\begin{equation}\label{eq:exponential-fam-I}
	\bar{I}_n(k;\eta) = \frac{n_k}{n}I(\eta), \quad \bar{I}(k;\eta) = I(\eta)\sum_{i=1}^{k}\gamma_k.
\end{equation}
Further, using \eqref{eq:exponential-family-expectation-of-suff-statistic},
\begin{align}\label{eq:exponential-fam-info-matrix}
	I(\eta) = \E{-\nabla S^\intercal(Y;\eta)} = \E{\nabla^2_\eta A(\eta)} = \nabla^2_\eta A(\eta) = \Varsub{\eta}{T(Y)}
\end{align} and plugging these into \eqref{eq:exponential-fam-I} gives
\begin{align}
\bar{I}_n(k;\eta) &= \frac{n_k}{n}\Varsub{\eta}{T(Y)}, \quad \bar{I}(k;\eta) = \Varsub{\eta}{T(Y)}\sum_{i=1}^{k}\gamma_k \label{eq:exponential-fam-I-bar}\\
\bar{I}_n(G_k;\eta) &= \frac{\abs{G_k}}{n}\Varsub{\eta}{T(Y)}, \quad \bar{I}(k;\eta) = \gamma_k \Varsub{\eta}{T(Y)}. \label{eq:exponential-fam-I-bar-groups}
\end{align}

Referring to \eqref{my_result_1_statement}, we start by simplifying 
\begin{align*}
	&K_1(\eta_0) = \sum_{k_1 = 1}^{K} \sum_{k_2 = k_1}^{\min\{k_1 + 1, K\}} \sum_{l = 1}^{d} \sum_{j = 1}^{d} \abs{\bar{I}_n^{-\frac{1}{2}}(G_{k_2}; \eta_0)_{lj}} \Bigg \{ \Bigg ( \nonumber \\
	&\sum_{i = 1}^{d} \sqrt{ \mathbb{E} [[\hat{\eta}_{[k_1]} - \eta_0]_i]^2 \mathbb{E} \Big (\frac{\partial^2}{\partial \eta_i \partial \eta_l} \ell_{k_1}(\eta_0, Y) + n\bar{I}_n(k_1, \eta_0)_{li}\Big)^2 }   \Bigg ) \nonumber \\
	&+ \frac{1}{2} \sum_{i = 1}^{d}\sum_{u = 1}^{d} \big (\mathbb{E} \big \{([\hat{\eta}_{[k_1]} - \eta_0]_i)^2 ([\hat{\eta}_{[k_1]} - \eta_0]_u)^2 \big \} \big )^{\frac{1}{2}} (\mathbb{E} [ \big (M^{k_1}_{iul}(Y) \big)^2 \mid \abs{Q_{(m)}} < \epsilon ] )^\frac{1}{2} \Bigg\}.
\end{align*}
By using \eqref{eq:exponential-fam-I-bar} and differentiating the log-likelihood, 
\begin{equation*}
	\frac{\partial^2}{\partial \eta_i \partial \eta_l} \ell_{k_1}(\eta_0, Y) = -n_k\frac{\partial^2}{\partial \eta_i \partial \eta_l} A(\eta) = -n\bar{I}_n(k_1, \eta_0)_{li}.
\end{equation*}
Thus,
\begin{equation}\label{eq:exponential-fam-K_1-I-bar}
	\mathbb{E} \Big (\frac{\partial^2}{\partial \eta_i \partial \eta_l} \ell_{k_1}(\eta_0, Y) + n\bar{I}_n(k_1, \eta_0)_{li}\Big)^2 = 0
\end{equation}
and so the first sum indexed by $i$ in $K_1$ is equal to $0$. We are left with simplifying $\mathbb{E} \big \{([\hat{\eta}_{[k_1]} - \eta_0]_i)^2 ([\hat{\eta}_{[k_1]} - \eta_0]_u)^2 \big \}$ and $\mathbb{E} [ \big (M^{k_1}_{iul}(Y) \big)^2 \mid \abs{Q_{(m)}} < \epsilon ]$. From the general form of the MLE of exponential families~\eqref{eq:exponential-fam-mle-special-form},
\begin{equation*}
	[\hat{\eta}_{[k_1]} - \eta_0]_i = \tau^{-1}\left (\frac{1}{n_{k_1}}\sum_{s=1}^{n_{k_1}}T(y_s) \right )_i - \eta_{0,i}.
\end{equation*}
Thus,
\begin{equation}\label{eq:exponential-fam-K_1-MSE}
\begin{split}
	\mathbb{E} \big \{([\hat{\eta}_{[k_1]} - \eta_0]_i)^2 ([\hat{\eta}_{[k_1]} - \eta_0]_u)^2 \big \} = \mathbb{E} \Bigg\{\left (\tau^{-1}\left (\frac{1}{n_{k_1}}\sum_{s=1}^{n_{k_1}}T(y_s) \right )_i - \eta_{0,i}\right )^2 \times \\ \left (\tau^{-1}\left (\frac{1}{n_{k_1}}\sum_{s=1}^{n_{k_1}}T(y_s) \right )_u - \eta_{0,u}\right )^2 \Bigg\}
\end{split}
\end{equation}
Without knowing the sufficient statistic $T(y)$ we cannot reduce this term further. Next we move on to the term involving the bounding function $M^{k_1}_{iul}(Y)$.

Taking a third partial derivative of the log-likelihood and recalling \eqref{eq:cumulant-3rd-derivative} we see that,
\begin{align*}
	\abs{\frac{\partial^3}{\partial \eta_i \partial \eta_u \partial \eta_l} \ell(\eta, y^{n_{k_1}})}
	= n_{k_1}\abs{\frac{\partial^3}{\partial \eta_i \partial \eta_u \partial \eta_l} A(\eta)} \\ = n_{k_1}\abs{(\mu_{ilk} - \mu_{il}\mu_{k} - \mu_{ik}\mu_l -\mu_{lk}\mu_i + 2\mu_i\mu_l \mu_{k})}.
\end{align*}
Note that the right hand side of the above equation is a function of $\eta$ only and not the observations $Y_1, \dots, Y_{n_k}$. 
Then,
\begin{equation}\label{eq:exponential-fam-M}
	(\mathbb{E} [ \big (M^{k_1}_{iul}(Y) \big)^2 \mid \abs{Q_{(m)}} < \epsilon ] )^\frac{1}{2} = n_{k_1} \mu^{\epsilon}_{iul},
\end{equation} where the latter is given by \eqref{eq:exp-fam-mu-epsilon}. Combining \eqref{eq:exponential-fam-K_1-I-bar}, \eqref{eq:exponential-fam-K_1-MSE}, \eqref{eq:exponential-fam-M}, and \eqref{eq:exponential-fam-I-bar-groups} we get,
\begin{equation}\label{eq:exponential-fam-K1}
\begin{split}
	K_1(\eta_0) = &\frac{1}{2} \sum_{k_1 = 1}^{K} \sum_{k_2 = k_1}^{\min\{k_1 + 1, K\}} n_{k_1} \left (\frac{\abs{G_{k_2}}}{n} \right)^{-\frac{1}{2}}\sum_{l = 1}^{d} \sum_{j = 1}^{d} \mathrm{Var}^{-\frac{1}{2}}\left[{T(Y)}\right]_{lj} \Bigg \{ \\
	& \sum_{i = 1}^{d}\sum_{u = 1}^{d} \mu^{\epsilon}_{iul} \Bigg (\mathbb{E} \Bigg[\left (\tau^{-1}\left (\frac{1}{n_{k_1}}\sum_{s=1}^{n_{k_1}}T(y_s) \right )_i - \eta_{0,i}\right )^2 \times \\ &\left (\tau^{-1}\left (\frac{1}{n_{k_1}}\sum_{s=1}^{n_{k_1}}T(y_s) \right )_u - \eta_{0,u}\right )^2 \Bigg] \Bigg )^{\frac{1}{2}} \Bigg\}.
\end{split}
\end{equation}

Now we consider $K_2(\eta)$ in \eqref{my_result_1_statement}. Recall that,
\begin{equation*}
	K_2(\eta) = \sum_{k = 1}^{K} \Big\{ \sum_{j = 1}^d \Big [ \sum_{i \in G_k} \Var{\xi_{ij}^2} \Big ]^{\frac{1}{2}} + 2 \sum_{i < j} \Big [ \sum_{v \in G_k} \mathrm{Var} [ \xi_{vi} \xi_{vj}] \Big ]^ {\frac{1}{2}} \Big \}.
\end{equation*}
A familiar calculation yields
\begin{gather*}
	\Var{\xi_{ij}^2} = n^{-2}\Var{(T_j(Y) - \Esub{\eta_0}{T_j(Y)})^2} \\
	\Var{\xi_{vi}\xi_{vj}} = n^{-2}\Var{(T_i(Y) - \Esub{\eta_0}{T_i(Y)})(T_j(Y) - \Esub{\eta_0}{T_j(Y)})}.
\end{gather*}
Therefore plugging into the above form of $K_2(\eta)$ yields 
\begin{equation}\label{eq:exponential-fam-K2}
\begin{split}
K_2(\eta) = \frac{1}{\sqrt{n}}\sum_{k = 1}^{K} \left(\frac{\abs{G_k}}{n}\right)^{\frac{1}{2}}\Big\{ \sum_{j = 1}^d \Big [ \Var{(T_j(Y) - \mu_j(\eta_0))^2} \Big ]^{\frac{1}{2}} + \\
2 \sum_{i < j} \Big [ \Var{(T_i(Y) - \mu_i(\eta_0))(T_j(Y) - \mu_j(\eta_0))} \Big ]^ {\frac{1}{2}} \Big \}.
\end{split}
\end{equation}

Finally we address $K_3(\eta)$,
\begin{equation*}
\sum_{i = 1}^{n} \mathbb{E} \Big [ \sum_{j = 1}^d \abs{(\xi'_{ij} - \xi_{ij})}  \Big ]^3.
\end{equation*}
As before, we see that
\begin{align*}
	\abs{\xi_{ij}' - \xi_{ij}} &= n^{-\frac{1}{2}}\abs{(T_j(Y') - \mu_j(\eta_0)) -(T_j(Y) - \mu_j(\eta_0))} \\
	&= \abs{T_j(Y') -T_j(Y)}.
\end{align*}
Hence $K_3(\eta_0)$ is equal to
\begin{equation}\label{eq:exp-fam-K3}
	\frac{1}{\sqrt{n}} \mathbb{E} \Big [ \sum_{j = 1}^d \abs{T_j(Y') -T_j(Y)}  \Big ]^3.
\end{equation}
Combining \eqref{eq:exponential-fam-K1}, \eqref{eq:exponential-fam-K2}, and \eqref{eq:exp-fam-K3} gives the corllary.
\end{document}